\documentclass[11pt,a4paper]{amsart}
\usepackage[T1]{fontenc}
\usepackage{lmodern}
\usepackage[utf8]{inputenc}
\usepackage[margin=1in]{geometry}
\usepackage{amsmath,amssymb,amsthm,mathtools}
\usepackage[expansion=false]{microtype}
\usepackage{cite}
\usepackage[hidelinks]{hyperref}
\usepackage{bookmark}
\allowdisplaybreaks[2]
\numberwithin{equation}{section}
\newtheorem{theorem}{Theorem}[section]
\newtheorem{proposition}[theorem]{Proposition}
\newtheorem{lemma}[theorem]{Lemma}
\newtheorem{corollary}[theorem]{Corollary}
\theoremstyle{definition}
\newtheorem{remark}[theorem]{Remark}
\newtheorem*{conjecture}{Conjecture}
\newcommand{\R}{\mathbb R}
\newcommand{\Z}{\mathbb Z}
\newcommand{\ii}{\mathrm i}
\newcommand{\dd}{\,\mathrm d}
\newcommand{\supp}{\operatorname{supp}}
\newcommand{\Imn}{\operatorname{Im}}
\newcommand{\Ren}{\operatorname{Re}}
\newcommand{\TV}{\operatorname{TV}}
\newcommand{\BB}{B}
\newcommand{\CC}{\mathcal C}
\newcommand{\DD}{\mathcal D}
\newcommand{\ip}[2]{\langle #1,#2\rangle}
\newcommand{\norm}[1]{\left\lVert#1\right\rVert}
\newcommand{\la}{\langle}
\newcommand{\ra}{\rangle}
\newcommand{\eps}{\varepsilon}
\title[Scattering for the three-dimensional cubic NLS]{Global well-posedness and scattering for the three-dimensional defocusing cubic Schr\"odinger equation in $H^s$, $s>\frac{1}{2}$}
\author{Qingtang Su}
\address{Qingtang Su
\newline \indent Academy of Mathematics and Systems Science, Chinese Academy of Sciences, Beijing, China.
\newline \indent Morningside Center of Mathematics, Beijing, China.}
\email{suqingtang@amss.ac.cn}

\author{Zehua Zhao}
\address{Zehua Zhao
\newline \indent School of Mathematics and Statistics, Beijing Institute of Technology,
\newline \indent Key Laboratory of Algebraic Lie Theory and Analysis, Ministry of Education,
Beijing 100081, China}
\email{zzh@bit.edu.cn}
\hypersetup{
 pdftitle={Global well-posedness and scattering for the three-dimensional defocusing cubic Schrodinger equation in Hs, s greater than 1/2},
 pdfauthor={Qingtang Su and Zehua Zhao},
 pdfsubject={Global well-posedness and scattering for the three-dimensional defocusing cubic NLS}
}
\begin{document}
\raggedbottom
\begin{abstract}
We prove global well-posedness and scattering for the three-dimensional
defocusing cubic nonlinear Schr\"odinger equation with arbitrary initial
data in $H^s(\R^3)$, $s>\frac{1}{2}$, without radial symmetry.
The proof combines the $I$-method with long-time $L^2_{t,x}$ bounds
for products of frequency-localized components of the nonlinear solution.
The key high--low estimate follows from a directional interaction identity
pairing high-frequency mass with low-frequency energy, together with
an induction on frequency.
\end{abstract}
\maketitle
\tableofcontents

\section{Introduction}\label{sec:intro}
\subsection{Background}
We study the three-dimensional defocusing cubic nonlinear
Schr\"odinger equation
\begin{equation}\label{eq:NLS}
 \begin{cases}
  (\ii\partial_t+\Delta)u=|u|^2u,\\
  u(0)=u_0\in H^s(\R^3).
 \end{cases}
\end{equation}
Equation \eqref{eq:NLS} is invariant under the scaling
\begin{equation}\label{eq:scaling}
 u_\lambda(t,x)=\lambda u(\lambda^2t,\lambda x),\qquad
 \|u_\lambda(0)\|_{\dot H^a}
 =\lambda^{a-\frac{1}{2}}\|u_0\|_{\dot H^a}.
\end{equation}
The scale-invariant space is $\dot H^{\frac{1}{2}}$, so $s=\frac{1}{2}$ is the scaling-critical regularity. For $s>\frac{1}{2}$ and $u_0\in H^s(\R^3)$, local well-posedness gives $T=T(\|u_0\|_{H^s})>0$ and a solution of~\eqref{eq:NLS} on $[0,T)$; see~\cite{Cazenave,CW,TaoBook}. Thus, if $[0,T_*)$ is the maximal forward lifespan and $T_*<\infty$, then
\begin{equation}\label{eq:blowupcriterion}
 \limsup_{t\to T_*^-}\|u(t)\|_{H^s}=\infty.
\end{equation}
At $s=\frac{1}{2}$, local well-posedness still gives a lifespan $[0,T)$ with $T=T(u_0)>0$, but the lifespan depends on the profile of the datum rather than only its $\dot H^{\frac{1}{2}}$ norm; see~\cite{CW,TaoBook}.

The global problem asks whether solutions exist for all time and approach free Schr\"odinger evolutions as $t\to\pm\infty$. We recall the conjecture, whose subcritical assertion is proved here.
\begin{conjecture}
For every $u_0\in H^s(\R^3)$ with $s>\frac{1}{2}$, the solution to
\eqref{eq:NLS} is global and scatters in $H^s$ in both time directions.
At the scaling-critical level, the analogous statement is expected for
$u_0\in\dot H^{\frac{1}{2}}(\R^3)$, with scattering in $\dot H^{\frac{1}{2}}$.
\end{conjecture}

Theorem~\ref{thm:main} proves the subcritical assertion without radial symmetry and gives a global $H^s$ bound uniform on bounded sets of initial data. The scaling-critical assertion remains outside the scope of this paper.

Mass and energy conservation are the starting point for the global theory of~\eqref{eq:NLS}. For sufficiently regular solutions,
\begin{equation}\label{eq:massenergy}
 M(u(t))=\int_{\R^3}|u(t,x)|^2\dd x=M(u_0),
\end{equation}
and
\begin{equation}\label{eq:energy}
 E(u(t))=\frac{1}{2}\int_{\R^3}|\nabla u(t,x)|^2\dd x
       +\frac{1}{4}\int_{\R^3}|u(t,x)|^4\dd x
       =E(u_0).
\end{equation}
Conservation, local well-posedness, and persistence of regularity give global well-posedness in $H^1$; see~\cite{Cazenave}. Scattering also requires spacetime control. Ginibre--Velo~\cite{GV} obtained scattering in the energy space using dispersive estimates and the Lin--Strauss Morawetz estimate~\cite{LinStrauss78}:
\begin{equation}\label{eq:introLSMorawetz}
 \int_J\int_{\R^3}\frac{|u(t,x)|^4}{|x|}\dd x\dd t
 \lesssim \sup_{t\in J}\|u(t)\|_{\dot H^{\frac{1}{2}}}^2
 \lesssim \|u_0\|_2\sup_{t\in J}\|\nabla u(t)\|_2.
\end{equation}

Below $H^1$, the energy may be infinite, so conservation cannot be applied directly. Bourgain~\cite{Bourgain} combined Fourier truncation with smoothing of the Duhamel term to prove global well-posedness for $s>\frac{11}{13}$. For radial data, the Morawetz estimate~\eqref{eq:introLSMorawetz} gave global well-posedness and scattering for $s>\frac{5}{7}$.

Colliander, Keel, Staffilani, Takaoka, and Tao introduced the $I$-method in~\cite{CKSTT02}, replacing exact energy conservation by an almost conservation law, and proved global well-posedness for $s>\frac{5}{6}$. Their interaction Morawetz estimate~\cite[Corollary~2.3]{CKSTT04} gives
\begin{equation}\label{eq:introInteractionMorawetz}
 \int_J\int_{\R^3}|u(t,x)|^4\dd x\dd t
 \lesssim \|u_0\|_2^2
      \sup_{t\in J}\|u(t)\|_{\dot H^{\frac{1}{2}}}^2.
\end{equation}
Together with the $I$-method and an almost conservation law, they proved
global well-posedness and scattering for $s>\frac{4}{5}$;
see \cite[Theorem~1.1]{CKSTT04}.

Dodson~\cite[Theorem~1.2]{Dodson13} combined interaction Morawetz with a linear--nonlinear decomposition to obtain global well-posedness and scattering for $s>\frac{5}{7}$. The decomposition follows Roy's approach~\cite{Roy09} for the cubic wave equation. Su~\cite[Theorem~1.2]{Su} added a resonance correction to the modified energy, as developed in~\cite{CKSTTRes08}, and reached $s>\frac{49}{74}$. At the critical regularity, Kenig--Merle~\cite[Theorem~1.1]{KM} showed by concentration compactness and rigidity that a uniform $\dot H^{\frac{1}{2}}$ bound on the maximal lifespan implies global existence and scattering in $\dot{H}^{\frac{1}{2}}$. For results under additional assumptions on the data, see~\cite{DodsonCritical,DodsonBesov,ShenWu}.

\subsection{The long-time obstruction}

The multiplier $I_N$ leaves frequencies below $N$ unchanged and smooths higher frequencies into the energy space; see~\eqref{eq:Idef}. The modified energy $E(I_Nu)$ can therefore be estimated for $\frac{1}{2}<s<1$, but its variation accumulates when time is divided into intervals $J_j=[t_j,t_{j+1}]$ on which a spacetime norm is small. The schematic bound on one such interval is
\begin{equation}\label{eq:introshortincrement}
 |E(I_Nu(t_{j+1}))-E(I_Nu(t_j))|
 \lesssim N^{-\alpha}.
\end{equation}
The number of pieces can depend on $N$. To track this cost, take $\|u_0\|_{H^s}\le M$ and choose $\lambda=d_{s,M}N^{-\frac{2(1-s)}{2s-1}}$, with $d_{s,M}>0$ small enough that $E(I_Nu_\lambda(0))\le\frac{1}{2}$. The rescaled mass parameter $\mu:=\max\{1,\|u_\lambda(0)\|_2\}$ then satisfies
\[
 \mu\lesssim_{s,M}N^{(1-s)/(2s-1)}.
\]
On an interval with $E(I_Nu_\lambda)\le1$, interaction Morawetz~\eqref{eq:introInteractionMorawetz} gives $\|u_\lambda\|_{L^4_{t,x}}^4\lesssim_s\mu^3$, hence a partition into $O(\mu^3)$ intervals on which the quartic spacetime norm is small. Summing the increments yields a bound of size $\mu^3N^{-\alpha}$. The resulting power of $N$ is negative precisely when
\[
 \alpha>\frac{3(1-s)}{2s-1}.
\]
No fixed $\alpha>0$ meets this condition for every $s>\frac{1}{2}$. We therefore seek an energy-increment estimate on the whole interval $J$ that avoids the factor $\mu^3$ from summing short intervals.

For radial data, Dodson~\cite[Theorem~4.1]{Dodson19} proves that, on an interval $J\ni0$ with $E(I_Nu(t))\le1$, a sufficiently large $N$ depending on $s$ and $\norm{u_0}_{H^s}$ satisfies
\begin{equation}\label{eq:dodsonlongtime}
 \|P_{>N/8}\nabla I_Nu\|_{L^2_tL^6_x(J\times\R^3)}\lesssim 1;
\end{equation}
the bound is independent of the length of $J$. Combining this estimate with the modified-energy bounds, Dodson proved global well-posedness and scattering for radial data throughout $s>\frac{1}{2}$.

\subsection{Main result}
The radial Sobolev estimates used to prove~\eqref{eq:dodsonlongtime} do not extend directly to nonradial data. We instead estimate products of high- and low-frequency components of the nonlinear solution in $L_{t,x}^2$ over the whole interval. Frequency induction and the $I$-method then give the following result.
\begin{theorem}\label{thm:main}
Let $s>\frac{1}{2}$. The initial value problem \eqref{eq:NLS}
is globally well posed for every $u_0\in H^s(\R^3)$.
For every $M>0$, the corresponding solutions with
$\|u_0\|_{H^s}\le M$ satisfy
\begin{equation}\label{eq:uniformball}
 \sup_{t\in\R}\|u(t)\|_{H^s(\R^3)}\le C(s,M).
\end{equation}
Moreover, each solution scatters in both time directions:
there exist $u_\pm\in H^s(\R^3)$ such that
\begin{equation}\label{eq:scattering}
 \lim_{t\to\pm\infty}
 \|u(t)-e^{\ii t\Delta}u_\pm\|_{H^s(\R^3)}=0.
\end{equation}
\end{theorem}

Here well-posedness is understood in the local Strichartz class
$C(\R;H^s)\cap L^2_{\mathrm{loc}}(\R;W^{s,6})$.
The scattering-state maps $u_0\mapsto u_\pm$ are also continuous in $H^s$;
see Section~\ref{sec:completion}.

\begin{remark}
Theorem~\ref{thm:main} does not include $s=\frac{1}{2}$.
The proof uses $s>\frac{1}{2}$ both to normalize the modified energy
with quantitative control of the rescaled mass and to obtain the
high-frequency decay needed to absorb logarithmic losses.
These estimates are not uniform as $s\to\frac{1}{2}$ from above.
Consequently, the present argument does not yield an a priori
$\dot H^{\frac{1}{2}}$ bound at the endpoint, as required by the
Kenig--Merle scattering criterion~\cite[Theorem~1.1]{KM}.

\end{remark}
\begin{remark}\label{rem:decay}
Following \cite{FSZ}, if, in addition, $u_0\in L^1(\R^3)$, then the solution in
Theorem~\ref{thm:main} satisfies
\[
 \|u(t)\|_\infty\le C(u_0)|t|^{-\frac{3}{2}},\qquad t\ne0.
\]
\end{remark}
\subsection{Proof strategy}
We first bound the energy increment in terms of a bilinear estimate for the solution. The nonlinear terms in that estimate lead to a recurrence at lower frequency thresholds, which we close by induction.

\noindent\textbf{1. Modified energy and the polynomial loss.} Fix
$\frac{1}{2}<s<1$ and write $u$ for the rescaled solution. On a
finite interval $J\ni0$ with $\sup_JE(I_Nu)\le1$, set $U=I_Nu$ and
$F(u)=|u|^2u$. The modified energy satisfies
\[
 \frac{\dd}{\dd t}E(U)=\mathcal E_4+\mathcal E_6,\qquad
 \mathcal E_4=\Imn\int\Delta\overline U\,[F(U)-I_NF(u)];
\]
see \eqref{eq:exactenergy}. The quartic term gives the leading
contribution to the almost conservation estimate. Write $u_H=P_Hu$ and
$m_H=\min\{1,(N/H)^{1-s}\}$; the energy bound gives
$\|u_H\|_{L^\infty_tL^2_x}\lesssim_s(Hm_H)^{-1}$. A nonzero quartic
contribution has dyadic frequencies $H\sim H'\ge K_1\ge K_2$ with
$H'\ge c_EN$, where $c_E$ is fixed as in \eqref{eq:energythreshold}.
Its multiplier has a convolution kernel of total variation
$O_s(H^2m_H^2)$; the dyadic piece $\mathcal E_{4;\vec H}$,
$\vec H=(H,H',K_1,K_2)$, therefore satisfies
\[
\begin{split}
 \int_J|\mathcal E_{4;\vec H}(t)|\dd t
 \lesssim_s H^2m_H^2
 &\sup_y\|u_H(t,x+y)u_{K_1}(t,x)\|_{L^2_{t,x}(J)}\\
 {}\times{}&\sup_z\|u_{H'}(t,x+z)u_{K_2}(t,x)\|_{L^2_{t,x}(J)},
\end{split}
\]
where the fixed relative translations $y,z$ come from the multiplier kernel. Lemma~\ref{lem:coarse} gives~\eqref{eq:roughA} with $A=C_s(1+\mu^3)$ and $\mu\ge\max\{1,\|u(0)\|_2\}$. The forced bilinear estimate then yields $\sup_y\|u_H(t,x+y)u_K(t,x)\|_{L^2_{t,x}(J)}\lesssim_sA^2/(H^{\frac{3}{2}}m_Hm_K)$ for $K\ll H$. Applying the preliminary product bounds in both frequency regimes retains the loss $A^4$ in the energy increment.

Dodson exploits smoothing of the nonlinear part~\cite[Section~5]{Dodson13}, and Su also uses a correction to the energy~\cite[Section~5.1]{Su}. We instead retain $E(I_Nu)$ and improve the product estimate for the full nonlinear solution.

\noindent\textbf{2. Long-time bilinear estimates.} Consider two model free packets with carrier frequencies of sizes $H$ and $K$, where $H\gg K$, common envelope width $K^{-1}$, and unit $L^2$ norm. Their relative group velocity is comparable to $H$, so a crossing lasts $O((HK)^{-1})$. On the overlap of volume $O(K^{-3})$, each packet has amplitude $O(K^{\frac{3}{2}})$, giving a contribution $O(K^6\cdot K^{-3}\cdot(HK)^{-1})=O(K^2/H)$ to the squared $L^2_{t,x}$ norm. Multiplying by the squared $L^2$ bounds for $u_H$ and $u_K$ gives $\frac{K^2}{H}(Hm_H)^{-2}(Km_K)^{-2}=(H^3m_H^2m_K^2)^{-1}$.
The model suggests the same frequency weights for nonlinear solutions if their localized forcing can be controlled. Set $\ell_H=1+\log(2+A\mu H^2)$. For $R\ge1$, let $\BB(R)\ge1$ be the least constant such that, for dyadic $H,K$,
\begin{equation}\label{eq:introBdef}
 \sup_{x_0\in\R^3}
 \|u_H(t,x+x_0)u_K(t,x)\|_{L^2_{t,x}(J)}^2
 \le\BB(R)\frac{\ell_H}{H^3m_H^2m_K^2},
 \qquad H\ge R,\quad K\le H;
\end{equation}
the factor $\ell_H$ accommodates logarithms from the transverse trace estimate and the low-frequency sum. This definition covers both regimes. For $K\sim H$, H\"older and interaction Morawetz for the localized equation give the required product bound while retaining its nonlinear forcing; see Corollary~\ref{cor:diagonal}. Steps 3 and 4 treat separated frequencies $K\ll H$.

The preliminary bound $\BB(R)\lesssim_sA^4$ guarantees finiteness and supplies the terminal estimate for induction. Lemma~\ref{lem:pair} adds a gain at very small $K$ so that the low-frequency sums converge. Section~\ref{sec:energy} gives the conditional whole-interval estimate
\begin{equation}\label{eq:introenergy}
 \int_J\left|\frac{\dd}{\dd t}E(I_Nu(t))\right|\dd t
 \lesssim_s\BB(c_EN)
       \left(\frac{\ell_N^3}{N}+\frac{\ell_N}{N^2}\right),
\end{equation}
where the two terms bound the quartic and sextic variations, respectively. The sextic estimate uses the multiplier gain in Lemma~\ref{lem:compressed}. The remaining task is to prove $\BB(c_EN)\lesssim_s1$.

\noindent\textbf{3. Directional interaction.} For separated frequencies, localize the high component to a cone: $h=Q_Hu$ has Fourier support in $\{\xi:\xi\cdot\omega\ge\gamma H,\ |\xi|\le C_hH\}$. Its positive longitudinal frequency suggests pairing it with low-frequency energy across hyperplanes perpendicular to $\omega$. The low cutoff must contain the factor to be estimated and stay above a fixed multiple of the induction threshold, even when the lower frequency is arbitrarily small.

Fix $N^{\frac{1}{2}}\le R\le N$ and $H\ge R$, with $N$ and $H/K$ sufficiently large. Choose $L$ by~\eqref{eq:Lchoice}, so that $K\lesssim L\ll H$, $m_L\sim_sm_K$, and $L\gtrsim R$.

Set $w=P_{\le L}u$ and $v=P_{\le L}w=P_{\le L}^2u$, with $P_{\le L}^2\ne P_{\le L}$. Moving the projection across the kinetic term produces derivatives of $v$, so we use the same field in the potential of $e_L=\mu^{-2}|w|^2+|\nabla w|^2+\frac{1}{2}|v|^4$. The modified energy remains $E(I_Nu)$. The auxiliary energy $E_L=\sup_J\int e_L$ is $O_s(m_L^{-2})$; pairing with low-frequency mass would introduce $\mu^2$.

To obtain a weight with narrow Fourier support, smooth $|h|^2$ in the direction $\omega$ with $k_r$ from~\eqref{eq:densitykernel}. Take $r=\varepsilon L$ with fixed small $\varepsilon>0$. Positivity and unit integral preserve the sign and mass of the density, while $\supp\widehat{k_r}\subset[-2r,2r]$ gives the required frequency localization. Set $\rho_h=k_r*_\omega|h|^2$ and $M_h=\sup_J\|h\|_2^2$. For fixed $\eta\in\R$, define
\[
\begin{gathered}
 \mathcal Q_\eta(t)
 =\iint_{(x-y)\cdot\omega>\eta}\rho_h(t,x)e_L(t,y)\dd x\dd y,\\
 0\le\mathcal Q_\eta(t)\le M_hE_L;
\end{gathered}
\]
see~\eqref{eq:relativeaction}. This integral counts pairs for which the high-frequency mass lies more than $\eta$ ahead of the low-frequency energy in the direction $\omega$.

The local nonlinear contributions from $|\nabla w|^2$ and $\frac{1}{2}|v|^4$ cancel. Besides the transport terms, the differentiated energy contains commutators and the forcing error $P_{\le L}(F(u)-F(v))$.

Differentiating $\mathcal Q_\eta$ and applying Plancherel in the $\omega$ direction gives a positive principal term with coefficient comparable to $H$. The commutators cost the low-frequency energy and are absorbed under $H\ge C(L+E_L)$, as ensured by the choice of $L$ and sufficiently large $N$.

The endpoint action is bounded by $M_hE_L$. Dividing by the frequency gap after time integration gives $H^{-1}\|h\|_{L^\infty_tL^2_x}^2E_L\lesssim_sH^{-3}m_H^{-2}m_L^{-2}$. The transverse trace estimate then recovers the product $hu_K$ with logarithmic cost $\ell_H$.

\noindent\textbf{4. Nonlinear estimates and frequency recurrence.}
The remaining forcing terms are $G_H=Q_HF(u)$ in the equation for $h$ and $R_L=P_{\le L}(F(u)-F(v))$ in the equation for $w$.

Since $u-v$ has no frequencies below $L$, the narrow Fourier support of the weight forces two comparable largest frequencies in every nonzero weighted low-frequency contribution. The quartic terms involving $h$ and $G_H$ have the same two-pair structure. Each bilinear estimate contributes $\BB^{\frac{1}{2}}$, so the error remains linear in $\BB$. For sextic low-frequency errors, the two lowest frequencies are placed in $L^\infty_{t,x}$. With $\delta_T=(Tm_T^2)^{-1}$, Proposition~\ref{prop:actualflux} and~\eqref{eq:normalization} give
\[
\begin{split}
 &\frac{H^3m_H^2m_K^2}{\ell_H}
   \sup_{x_0\in\R^3}\|h(t,x+x_0)u_K(t,x)\|_{L^2_{t,x}(J)}^2\\
 &\qquad\lesssim_s 1+\BB(cH)\delta_H\ell_H^3
          +\BB(cL)\delta_L\ell_L^3(1+\delta_L).
\end{split}
\]
The three terms come from the endpoint action, $G_H$, and $R_L$. Because $H,L\gtrsim R$, the nonlinear coefficients are at most $C_s\ell_N^3/R$. Monotonicity of $\BB$ and $cH,cL\ge c_1R$, for fixed $c_1\in(0,1)$, give $\BB(cH),\BB(cL)\le\BB(c_1R)$. Combining the two frequency regimes and summing the fixed finite angular partition yields
\[
 \BB(R)\le C_s+C_s\frac{\ell_N^3}{R}\BB(c_1R),
 \qquad N^{\frac{1}{2}}\le R\le N;
\]
see Proposition~\ref{prop:recursion}.

\noindent\textbf{5. Iteration and energy closure.} Section~\ref{sec:completion} chooses $A+\mu$ polynomially bounded in $N$, so $\ell_N=O_{s,M}(\log N)$. Throughout the induction, the solution, interval, and these parameters stay fixed; only the threshold descends from $R=c_EN$ to below $N^{\frac{1}{2}}$. There are $\sim\log N$ steps, each with coefficient at most $C_s\ell_N^3N^{-\frac{1}{2}}$. The product of these coefficients decays faster than every fixed inverse power of $N$, suppressing the terminal bound $\BB\lesssim_sA^4$. The additive terms sum geometrically, giving $\BB(c_EN)\lesssim_s1$.

Starting from $E(I_Nu(0))\le\frac{1}{2}$, choose one sufficiently large $N=N(s,M)$ so that~\eqref{eq:introenergy} is smaller than $\frac{1}{4}$. This improves the energy bound from one to $\frac{3}{4}$, and continuity gives global modified-energy control. Undoing the scaling yields the uniform $H^s$ bound. Interaction Morawetz and Strichartz then give $u\in L^\infty(\R;H^s)\cap L^2(\R;W^{s,6})$ for the original solution. The nonlinear estimate makes the Duhamel tails vanish in $H^s$, proving scattering in both time directions.

The interaction approach developed here also motivated our subsequent
work on scattering below the ground state for the three-dimensional
focusing energy-critical NLS~\cite{SuZhaoFocusing}.

\begin{remark}[Further applications]\label{rem:two-dimensional}
For the nonlinearities $F(u)=|u|^{2k}u$ with integer $k\ge2$,
the cancellation in the auxiliary low-frequency energy persists
with potential density $|v|^{2k+2}/(k+1)$.
We expect that, with suitable modifications, the present approach
can be adapted to the nonradial counterpart of the two-dimensional
results in \cite[Theorem~1.4]{Dodson19}: global well-posedness and
scattering for initial data in $H^s(\R^2)$, $s>1-\frac{1}{k}$.
This would require suitable two-dimensional trace and bilinear
estimates, together with bounds for the higher-order nonlinear terms.

We also expect that suitable modifications of these ideas could
yield global well-posedness for the three-dimensional defocusing
cubic wave equation
\[
 \partial_t^2u-\Delta u+u^3=0
\]
with arbitrary initial data in
$H^s(\R^3)\times H^{s-1}(\R^3)$, $s>\frac{1}{2}$,
without radial symmetry.
For a related radial global well-posedness result, see
Dodson \cite{DodsonWave19}.
Such an extension would require interaction identities and bilinear
estimates adapted to wave propagation.
We do not pursue these extensions here.
\end{remark}

\subsection{Organization}
Section~\ref{sec:prelim} gives preliminary bounds for $u$ and $\BB$, and Section~\ref{sec:energy} expresses the modified-energy increment in terms of the bilinear constant. Section~\ref{sec:flux} proves the directional interaction estimate. Section~\ref{sec:recursion} controls its nonlinear terms and closes the frequency induction. Finally, Section~\ref{sec:completion} combines scaling, the almost conservation estimate, and local stability to prove Theorem~\ref{thm:main}.

\subsection{Notation}\label{sec:notation}
We write $X\lesssim Y$ if $X\le CY$ for a constant depending only on
$s$ and the fixed cutoff functions, unless other dependencies are
indicated. The notation $X\sim Y$ means $X\lesssim Y\lesssim X$.
We fix the cutoff functions independently of the frequency parameters
and the solution. All frequency comparisons use fixed constants.
All frequency sums are over dyadic numbers in $2^{\Z}$.
For a finite complex measure $\nu$ on $\R^d$, we write $|\nu|$ for
its total variation measure and $\|\nu\|_{\TV}=|\nu|(\R^d)$ for
its total variation norm.
For an interval $J$, mixed norms are taken over $J\times\R^3$;
we abbreviate $L^p_tL^p_x$ to $L^p_{t,x}$. We write $\|f\|_p$ for
the spatial $L^p$ norm and set $\langle\nabla\rangle=(1-\Delta)^{\frac{1}{2}}$.
We use the Fourier transform
and complex inner product
\[
 \widehat f(\xi)=\int_{\R^3}e^{-2\pi\ii x\cdot\xi}f(x)\dd x,
 \qquad \ip fg=\int_{\R^3}\overline f g\dd x.
\]
We set $D=(2\pi\ii)^{-1}\nabla$, so that
$\widehat{a(D)f}(\xi)=a(\xi)\widehat f(\xi)$.
The same Fourier convention is used in lower-dimensional variables.
Choose a real radial function $\chi\in C_c^\infty(\R^3)$ with
$0\le\chi\le1$, equal to one on $|\xi|\le1$ and zero on $|\xi|\ge2$.
The homogeneous Littlewood--Paley projections are defined by
\[
 \widehat{P_Hf}(\xi)
 =\bigl[\chi(\xi/H)-\chi(2\xi/H)\bigr]\widehat f(\xi),
 \qquad f_H=P_Hf.
\]
Thus $\sum_HP_H=1$ away from the origin and
$\supp\widehat f_H\subset\{H/2\le|\xi|\le2H\}$.
For any $\rho>0$, set
$P_{\le\rho}=\chi(D/\rho)$, $P_{>\rho}=1-P_{\le\rho}$, and
$P_{\ge\rho}=1-P_{\le\rho/2}$. We write
$f_{\le\rho}=P_{\le\rho}f$ and similarly for the other cutoffs.
As usual, these are smooth Fourier cutoffs, not idempotent projections.
We use standard Bernstein and
Littlewood--Paley estimates throughout; see \cite[Appendix~A]{TaoBook}.
All relative spatial translations in spacetime norms are independent
of time.

In the multilinear arguments, we first truncate the dyadic expansions.
The summable bounds proved below justify removal of these truncations
by dominated convergence. For the bilinear estimates, this limit is
taken at each fixed relative translation before taking the supremum
over translations.

\section{Preliminary estimates}\label{sec:prelim}
The local theory and interaction Morawetz give finite spacetime bounds under $\sup_J E(I_Nu)\le1$. These bounds depend polynomially on the mass parameter and supply the terminal estimate for frequency induction. We also record the linear, bilinear, and dyadic estimates used to express the energy increment in Section~\ref{sec:energy}.

\subsection{Fourier multipliers and the \texorpdfstring{$I$}{I}-operator}\label{sec:prelim-multipliers}
Fix $\frac{1}{2}<s<1$. Following \cite[Definition~2.2]{Su}, for $N\ge1$ define
$I_N:H^s(\R^3)\to H^1(\R^3)$ by
\begin{equation}\label{eq:Idef}
 \widehat{I_Nf}(\xi)=m_N(\xi)\widehat f(\xi),\qquad
 m_N(\xi)=
 \begin{cases}
  1,&|\xi|\le N,\\
  (N/|\xi|)^{1-s},&|\xi|\ge2N,
 \end{cases}
\end{equation}
where $m_N(\xi)=m_1(\xi/N)$ for a fixed positive, smooth, radial
function $m_1$ that is nonincreasing in $|\xi|$. We write $I=I_N$.
The associated modified energy is
\begin{equation}\label{eq:modifiedenergy}
 E(I_Nu(t))
 =\frac{1}{2}\|\nabla I_Nu(t)\|_{L^2_x}^2
  +\frac{1}{4}\|I_Nu(t)\|_{L^4_x}^4.
\end{equation}
For dyadic $H>0$, set
\begin{equation}\label{eq:weights}
 m_H=\min\{1,(N/H)^{1-s}\}.
\end{equation}
Then $m_N(\xi)\sim_s m_H$ for $H/2\le|\xi|\le2H$.

For $b\in C_c^6(B(0,R_0))$ and $\rho>0$, integration by parts
and scaling give
\begin{equation}\label{eq:smoothkernel}
 K_{b(D/\rho)}(x)=\rho^3\check b(\rho x),\qquad
 \|K_{b(D/\rho)}\|_1\le C_{R_0}\|b\|_{W^{6,\infty}}.
\end{equation}
Indeed, $|\check b(x)|\lesssim_{R_0}\|b\|_{W^{6,\infty}}(1+|x|)^{-6}$.
Thus the kernel bound is uniform for symbols supported in a common
ball with bounded $W^{6,\infty}$ norms.

Let $p_H(\xi)=\chi(\xi/H)-\chi(2\xi/H)$ be the symbol of $P_H$.
For $1\le j\le3$, define
\[
 \widehat K_{H,j}(\xi)
 =\frac{p_H(\xi)\xi_j}{2\pi\ii|\xi|^2m_N(\xi)}.
\]
The definitions and \eqref{eq:smoothkernel} give
\begin{equation}\label{eq:dyadicIestimate}
 P_Hf=\sum_{j=1}^3K_{H,j}*\partial_jI_Nf,\qquad
 \sum_{j=1}^3\|K_{H,j}\|_{L^1}\lesssim_s(Hm_H)^{-1}.
\end{equation}
Indeed, $(Hm_H)\widehat K_{H,j}(H\xi)$ has uniformly bounded
$W^{6,\infty}$ norm and is supported in $\frac{1}{2}\le|\xi|\le2$.

To estimate $\nabla I_N(|u|^2u)$, we use the following $I_N$ version of the Kato--Ponce product rule~\cite{KP88}, whose constant is independent of $N$.

\begin{lemma}\label{lem:Iproduct}
Let $\frac{1}{2}<s<1$ and
\[
 1<r,p_1,q_1,p_2,q_2<\infty,\qquad
 \frac1r=\frac1{p_1}+\frac1{q_1}
        =\frac1{p_2}+\frac1{q_2}.
\]
For Schwartz functions $f,g$,
\begin{equation}\label{eq:Iproduct}
 \||\nabla|I_N(fg)\|_r\le C\left(
 \||\nabla|I_Nf\|_{p_1}\|g\|_{q_1}
 +\|f\|_{p_2}\||\nabla|I_Ng\|_{q_2}\right).
\end{equation}
Here $C$ depends on $s,r,p_1,q_1,p_2,q_2$ and the fixed cutoff
functions, but not on $N\ge1$.
\end{lemma}
\begin{proof}
For $\widetilde f_N(x)=f(x/N)$,
\[
 |\nabla|I_1\widetilde f_N(x)
 =N^{-1}(|\nabla|I_Nf)(x/N).
\]
Both sides of~\eqref{eq:Iproduct} scale by $N^{3/r-1}$, so it suffices to prove the estimate for $N=1$. Write $f_\ell=P_{\le1}f$, $f_h=P_{>1}f$, and use the same decomposition for $g$. The Fourier multiplier theorem gives, for $1<p<\infty$,
\begin{align*}
 \|\nabla f_\ell\|_p+\|\langle\nabla\rangle^s f_h\|_p
 &\lesssim_{s,p}\||\nabla|I_1f\|_p,&
 \|\langle\nabla\rangle^s f_\ell\|_p&\lesssim_{s,p}\|f\|_p,\\
 \||\nabla|I_1w\|_p&\lesssim_{s,p}\|\nabla w\|_p,&
 \||\nabla|I_1w\|_p&\lesssim_{s,p}\|\langle\nabla\rangle^s w\|_p.
\end{align*}
Thus the ordinary product rule bounds the low-frequency product by
\[
 \||\nabla|I_1(f_\ell g_\ell)\|_r
 \lesssim\||\nabla|I_1f\|_{p_1}\|g\|_{q_1}
        +\|f\|_{p_2}\||\nabla|I_1g\|_{q_2}.
\]
For the mixed term, the inhomogeneous Kato--Ponce inequality
\cite[Theorem~1]{GO} gives
\begin{align*}
 \||\nabla|I_1(f_hg_\ell)\|_r
 &\lesssim\|\langle\nabla\rangle^s f_h\|_{p_1}\|g_\ell\|_{q_1}
       +\|f_h\|_{p_1}\|\langle\nabla\rangle^s g_\ell\|_{q_1}\\
 &\lesssim\||\nabla|I_1f\|_{p_1}\|g\|_{q_1}.
\end{align*}
The term $f_\ell g_h$ is treated symmetrically. Applying the same
Kato--Ponce inequality to $f_hg_h$ gives the two terms on the
right of \eqref{eq:Iproduct}. Summing proves the result.
\end{proof}

\subsection{Strichartz and bilinear estimates}\label{sec:prelim-linear}
Write $S(t)=e^{\ii t\Delta}$, so that
$\widehat{S(t)f}(\xi)=e^{-4\pi^2\ii t|\xi|^2}\widehat f(\xi)$.
A pair $(q,r)$ is Schr\"odinger-admissible if
\[
 2\le q\le\infty,\qquad 2\le r\le6,\qquad
 \frac2q+\frac3r=\frac{3}{2}.
\]
For any two admissible pairs $(q,r)$ and $(\widetilde q,\widetilde r)$,
the Strichartz estimates (see \cite{KT} and the references therein) give, on $J=[t_0,t_1]$,
\begin{align}
 \|S(t-t_0)f_0\|_{L^q_tL^r_x(J)}
 &\lesssim\|f_0\|_2,\label{eq:homogeneousStrichartz}\\
 \left\|\int_{t_0}^tS(t-\tau)G(\tau)\dd\tau\right\|_{L^q_tL^r_x(J)}
 &\lesssim\|G\|_{L^{{\widetilde{q}}'}_tL^{{\widetilde{r}}'}_x(J)},
 \label{eq:linearStrichartz}\\
 \left\|\int_J S(-t)G(t)\dd t\right\|_2
 &\lesssim\|G\|_{L^{{\widetilde{q}}'}_tL^{{\widetilde{r}}'}_x(J)}.
 \label{eq:dualStrichartz}
\end{align}
The last estimate is the dual form of \eqref{eq:homogeneousStrichartz}.
The constants are independent of $J$. Set
\[
 \|f\|_{S^0(J)}=\|f\|_{L^\infty_tL^2_x(J)}
                    +\|f\|_{L^2_tL^6_x(J)}.
\]
By interpolation, $\|f\|_{L^q_tL^r_x(J)}\lesssim\|f\|_{S^0(J)}$
for every admissible pair $(q,r)$.
For the free Schr\"odinger evolution, we use the bilinear estimate
\cite[Theorem~4.18]{KV13}:
\begin{equation}\label{eq:freebilinear}
 \|S(t)f_H S(t)g_K\|_{L^2_{t,x}(\R^{1+3})}
 \lesssim KH^{-\frac{1}{2}}\|f_H\|_2\|g_K\|_2,
 \qquad K\ll H.
\end{equation}
For the earlier two-dimensional estimate, see \cite{Bourgain98Refined}.
The frequency component satisfies $(\ii\partial_t+\Delta)u_H=P_H(|u|^2u)$, so the free estimate must be extended to forced equations. We use the Christ--Kiselev argument from the proof of~\cite[Lemma~2.5]{Visan07}, starting from~\eqref{eq:freebilinear}, and include the derivation to record the forcing norms.

\begin{lemma}\label{lem:forcedbilinear}
Let $f_H,g_K$ solve
\[
 (\ii\partial_t+\Delta)f_H=F_H,\qquad
 (\ii\partial_t+\Delta)g_K=G_K
\]
on $J=[t_0,t_1]$, with $L^2$ initial data and
$F_H,G_K\in L^{\frac{3}{2}}_tL^{\frac{18}{13}}_x(J)$.
Assume that $f_H,F_H$ have spatial Fourier support in
$\{H/2\le|\xi|\le2H\}$, and $g_K,G_K$ in
$\{K/2\le|\xi|\le2K\}$. If $K\ll H$, then
\begin{equation}\label{eq:forcedbilinear}
 \begin{split}
 \|f_Hg_K\|_{L^2_{t,x}(J)}
 \lesssim KH^{-\frac{1}{2}}
 &\bigl(\|f_H(t_0)\|_2+\|F_H\|_{L^{\frac{3}{2}}_tL^{\frac{18}{13}}_x(J)}\bigr)\\
 {}\times&\bigl(\|g_K(t_0)\|_2+\|G_K\|_{L^{\frac{3}{2}}_tL^{\frac{18}{13}}_x(J)}\bigr).
 \end{split}
\end{equation}
The same bound holds after any fixed spatial translation of either factor.
\end{lemma}
\begin{proof}
We first replace the low-frequency free factor by a forced solution, then repeat the transfer for the high-frequency factor.

\emph{Step 1: One forced factor.} Let $h_0\in L^2(\R^3)$ have $\supp\widehat h_0\subset\{H/2\le|\xi|\le2H\}$, and set $h(t)=S(t-t_0)h_0$. Define
\[
 a_K=\int_J S(t_0-\tau)G_K(\tau)\dd\tau.
\]
The dual Strichartz estimate \eqref{eq:dualStrichartz}, with
$(\widetilde q,\widetilde r)=(3,\frac{18}{5})$, gives
\[
 \|a_K\|_2\lesssim\|G_K\|_{L^{\frac{3}{2}}_tL^{\frac{18}{13}}_x(J)}.
\]
Since $a_K$ has Fourier support in $\{K/2\le|\xi|\le2K\}$,
\eqref{eq:freebilinear} yields
\begin{align*}
 \left\|h(t)\int_J S(t-\tau)G_K(\tau)\dd\tau\right\|_{L^2_{t,x}(J)}
 &=\|S(t-t_0)h_0\,S(t-t_0)a_K\|_{L^2_{t,x}(J)}\\
 &\lesssim KH^{-\frac{1}{2}}\|h_0\|_2\|a_K\|_2\\
 &\lesssim KH^{-\frac{1}{2}}\|h_0\|_2
       \|G_K\|_{L^{\frac{3}{2}}_tL^{\frac{18}{13}}_x(J)}.
\end{align*}
For fixed $h_0$, apply the Banach-valued Christ--Kiselev lemma~\cite[Theorem~1.2]{ChristKiselev}. The strict time-exponent inequality $\frac{3}{2}<2$ allows the full integral over $J$ to be restricted to $[t_0,t]$. Adding the free evolution $S(t-t_0)g_K(t_0)$ gives
\begin{equation}\label{eq:freeforcedbilinear}
 \|S(t-t_0)h_0\,g_K\|_{L^2_{t,x}(J)}
 \lesssim KH^{-\frac{1}{2}}\|h_0\|_2
 \bigl(\|g_K(t_0)\|_2+\|G_K\|_{L^{\frac{3}{2}}_tL^{\frac{18}{13}}_x(J)}\bigr).
\end{equation}
\emph{Step 2: Two forced factors.} Now keep $g_K$ fixed. In~\eqref{eq:freeforcedbilinear}, take $h_0=\int_J S(t_0-\tau)F_H(\tau)\dd\tau$ and use~\eqref{eq:dualStrichartz}:
\begin{align*}
 &\left\|g_K(t)\int_J S(t-\tau)F_H(\tau)\dd\tau\right\|_{L^2_{t,x}(J)}\\
 &\qquad\lesssim KH^{-\frac{1}{2}}\|F_H\|_{L^{\frac{3}{2}}_tL^{\frac{18}{13}}_x(J)}
 \bigl(\|g_K(t_0)\|_2+\|G_K\|_{L^{\frac{3}{2}}_tL^{\frac{18}{13}}_x(J)}\bigr).
\end{align*}
A second application of Christ--Kiselev, again using $\frac{3}{2}<2$, restricts the high-frequency forcing integral to $[t_0,t]$. Duhamel's formula and~\eqref{eq:freeforcedbilinear} then give~\eqref{eq:forcedbilinear}. Fixed spatial translations preserve all norms and commute with $S(t)$.
\end{proof}

At a single frequency, Bernstein's inequality gives
\begin{equation}\label{eq:singleL4}
 \|f_H\|_{L^4_{t,x}(J)}
 \lesssim H^{\frac{1}{4}}\|f_H\|_{L^4_tL^3_x(J)}
 \lesssim H^{\frac{1}{4}}\|f_H\|_{S^0(J)}.
\end{equation}
\subsection{Preliminary bounds for the solution}\label{sec:prelim-bounds}
We work first with a Schwartz solution on a finite interval $J=[t_0,t_1]$.
Fix a sufficiently large dyadic $N$ and assume
\begin{equation}\label{eq:boot}
 \sup_{t\in J}E(I_Nu(t))\le1,\qquad
 \mu\ge\max\{1,\|u(t_0)\|_2\}.
\end{equation}
Write $F(u)=|u|^2u$.
We use the following local estimate from \cite[Theorem~3.2]{Dodson13}.

\begin{lemma}\label{lem:local}
There exist $\epsilon_s>0$, $N_s\ge1$, and $C_s>0$, depending only
on $s$ and the fixed cutoff functions, with the following property.
If $J'=[t_*,t^*]\subset J$, $N\ge N_s$, $E(I_Nu(t_*))\le1$, and
$\|u\|_{L^4_{t,x}(J')}\le\epsilon_s$, then
\begin{equation}\label{eq:localI}
 \|\nabla I_Nu\|_{S^0(J')}\le C_s.
\end{equation}
\end{lemma}

The energy bound gives pointwise frequency control. Interaction Morawetz bounds the number of intervals with small $L^4_{t,x}$ norm, yielding a coarse estimate on $J$. Set $b_H=(Hm_H)^{-1}$ for the frequency-localized bounds below.

\begin{lemma}\label{lem:coarse}
Under \eqref{eq:boot}, with $N\ge N_s$,
\begin{align}
 \|u_H\|_{L^\infty_tL^2_x}&\lesssim_sb_H,&
 \|u_H\|_{L^\infty_{t,x}}&\lesssim_s H^{\frac{1}{2}}m_H^{-1},
       \label{eq:massfreq}\\
 \|u_{\le N}\|_{L^\infty_t\dot H^{\frac{1}{2}}_x}^2&\lesssim_s\mu,&
 \|u_{>N}\|_{L^\infty_t\dot H^{\frac{1}{2}}_x}&\lesssim_sN^{-\frac{1}{2}},
       \label{eq:criticalrough}\\
 \|u\|_{L^4_{t,x}(J)}^4&\lesssim_s1+\mu^3.\label{eq:Morawetzrough}
\end{align}
Moreover, one may choose $A=C_s(1+\mu^3)\ge2$ so that
\begin{equation}\label{eq:roughA}
 \|\nabla I_Nu\|_{S^0(J)}
 +\|\nabla I_NF(u)\|_{L^{\frac{3}{2}}_tL^{\frac{18}{13}}_x(J)}\le A.
\end{equation}
The same $A$ applies on every subinterval of $J$.
\end{lemma}
\begin{proof}
The first two bounds follow from \eqref{eq:dyadicIestimate}, the
energy bound, and Bernstein's inequality. Interpolation below $N$
and the definition of $m_N$ above $N$ give
\[
 \|u_{\le N}\|_{\dot H^{\frac{1}{2}}}^2
 \lesssim_s\|u\|_2\|\nabla I_Nu\|_2\lesssim_s\mu,
 \qquad
 \sup_{|\xi|\gtrsim N}\frac{|\xi|^{\frac{1}{2}}}{|\xi|m_N(\xi)}
 \lesssim_sN^{-\frac{1}{2}}.
\]
The interaction Morawetz estimate \eqref{eq:introInteractionMorawetz} yields
\[
 \|u\|_{L^4_{t,x}(J)}^4
 \lesssim\|u\|_{L^\infty_tL^2_x}^2
             \|u\|_{L^\infty_t\dot H^{\frac{1}{2}}_x}^2
 \lesssim_s\mu^2(\mu+N^{-1})\lesssim_s1+\mu^3.
\]
Partition $J$ into $n\lesssim_s1+\mu^3$ intervals $J_\nu$ satisfying $\|u\|_{L^4_{t,x}(J_\nu)}\le\epsilon_s$. Lemma~\ref{lem:local} gives $\|\nabla I_Nu\|_{S^0(J_\nu)}\lesssim_s1$ on each piece. The auxiliary spacetime estimate~\cite[Lemma~3.1]{Dodson13} gives
\[
 Y_\nu:=\|u\|_{L^6_tL^{\frac{9}{2}}_x(J_\nu)}
 \lesssim_s(\epsilon_s^{\frac{2}{3}}+N^{-\frac{1}{2}})
       (1+\|\nabla I_Nu\|_{S^0(J_\nu)})
 \lesssim_s\epsilon_s^{\frac{2}{3}}+N^{-\frac{1}{2}}.
\]
Lemma~\ref{lem:Iproduct}, the Riesz transform bounds, and H\"older's
inequality give
\begin{equation}\label{eq:localcubic}
 \begin{split}
 \|\nabla I_NF(u)\|_{L^{\frac{3}{2}}_tL^{\frac{18}{13}}_x(J_\nu)}
 &\lesssim_s\|\nabla I_Nu\|_{L^3_tL^{\frac{18}{5}}_x(J_\nu)}Y_\nu^2\\
 &\lesssim_s\epsilon_s^{\frac{4}{3}}+N^{-1}\lesssim_s1.
 \end{split}
\end{equation}
Summing over the subintervals gives
\[
 \|\nabla I_Nu\|_{S^0(J)}\lesssim_s1+n^{\frac{1}{2}},\qquad
 \|\nabla I_NF(u)\|_{L^{\frac{3}{2}}_tL^{\frac{18}{13}}_x(J)}\lesssim_s n^{\frac{2}{3}}.
\]
Increasing $C_s$ proves \eqref{eq:roughA}.
Both norms decrease on subintervals of $J$.
\end{proof}

\subsection{Bilinear bounds and dyadic summation}\label{sec:prelim-bilinear}
The energy identity in Section~\ref{sec:energy} contains products of four or six frequency components. We pair each of the two largest frequencies with a lower frequency. For $K\ll H$,~\eqref{eq:freebilinear} and the frequency-localized $L^2$ bounds~\eqref{eq:massfreq} suggest the weight $(H^{\frac{3}{2}}m_Hm_K)^{-1}$; for $K\sim H$,~\eqref{eq:singleL4} gives the same frequency powers. Applying Lemma~\ref{lem:forcedbilinear} and~\eqref{eq:singleL4} to the nonlinear solution with the preliminary bound~\eqref{eq:roughA} introduces a factor $A^2$; see Lemma~\ref{lem:pair}. Relative translations are retained because the energy multipliers are represented by convolution kernels.

For $H>0$ set
\begin{equation}\label{eq:log}
 \ell_H=1+\log(2+A\mu H^2).
\end{equation}
The factor $\ell_H$ accommodates the logarithms in the low-frequency summation of Lemma~\ref{lem:sums} and the transverse trace estimate of Lemma~\ref{lem:trace}, with the parameters chosen in Section~\ref{sec:recursion}. For $R\ge1$, the least constant introduced in~\eqref{eq:introBdef} is
\begin{equation}\label{eq:Bdef}
 \BB(R)=\max\left\{1,
 \sup_{\substack{H\ge R,\ K\le H\\x_0\in\R^3}}
 \frac{H^3m_H^2m_K^2}{\ell_H}
 \norm{u_H(t,x+x_0)u_K(t,x)}_{L^2_{t,x}(J)}^2\right\}.
\end{equation}
Here $u$, $J$, and $N,A,\mu$ are fixed. The supremum ranges over dyadic frequencies and time-independent translations, so $\BB$ is nonincreasing in $R$. No positive lower bound is imposed on $K$, and no upper bound is imposed on $H$.

\begin{lemma}\label{lem:pair}
Under \eqref{eq:boot} and \eqref{eq:roughA},
\begin{equation}\label{eq:Brough}
 \BB(R)\lesssim_s A^4.
\end{equation}
For $H\ge R\ge1$ and $K\le H$, uniformly in $x_0$,
\begin{equation}\label{eq:pair}
 \norm{u_H(t,x+x_0)u_K(t,x)}_{L^2_{t,x}}
 \lesssim_s
 \frac{\BB(R)^{\frac{1}{2}}\ell_H^{\frac{1}{2}}}{H^{\frac{3}{2}}m_Hm_K}
 \theta_{H,K},
 \qquad
 \theta_{H,K}=\min\{1,C_sA\mu\sqrt{HK}\}.
\end{equation}
The constant $C_s$ in $\theta_{H,K}$ can be fixed sufficiently large once and for all.
\end{lemma}
\begin{proof}
Write $F_H=P_H(|u|^2u)$. Applying \eqref{eq:dyadicIestimate}
to \eqref{eq:roughA} gives
\[
 \|u_H\|_{S^0(J)}+\|u_H(t_0)\|_2+\|F_H\|_{L^{\frac{3}{2}}_tL^{\frac{18}{13}}_x(J)}
 \lesssim_s Ab_H.
\]
For $K\ll H$, Lemma~\ref{lem:forcedbilinear} yields
\[
 \|u_H(\cdot+x_0)u_K\|_{L^2_{t,x}}
 \lesssim_s KH^{-\frac{1}{2}}(Ab_H)(Ab_K)
 =\frac{A^2}{H^{\frac{3}{2}}m_Hm_K}.
\]
For $K\sim H$, the same bound follows from \eqref{eq:singleL4}
and H\"older's inequality. Consequently $\BB(R)\lesssim_s A^4$.
For $K<1$, mass conservation and Bernstein's inequality give
\begin{equation}\label{eq:tiny}
 \|u_H(\cdot+x_0)u_K\|_{L^2_{t,x}}^2
 \le\|u_H\|_{L^2_tL^6_x}^2\|u_K\|_{L^\infty_tL^3_x}^2
 \lesssim_s\frac{A^2\mu^2K}{H^2m_H^2}.
\end{equation}
Taking the minimum of~\eqref{eq:tiny} and the defining bound for $\BB(R)$, and using $m_K=1$ and $\BB(R)\ell_H\ge1$, gives~\eqref{eq:pair}. The resulting positive power of $K$ makes the low-frequency sum convergent. For $K\ge1$, $\theta_{H,K}=1$ and~\eqref{eq:Bdef} gives the bound directly.
\end{proof}

For $H'\sim H$ and $K'\sim K$, the definitions give
\begin{equation}\label{eq:thetastability}
 m_{H'}\sim_s m_H,\qquad \ell_{H'}\sim\ell_H,\qquad
 \theta_{H',K'}\sim\theta_{H,K}.
\end{equation}
The factor $\theta_{H,K}$ makes the low-frequency sums in the energy estimate convergent. The nonlinear estimates in Section~\ref{sec:recursion} also require geometric decay at high frequencies. Set
\begin{equation}\label{eq:deltadef}
 \delta_H=\frac1{Hm_H^2}
 =\begin{cases}H^{-1},&H\le N,\\
 N^{2s-2}H^{1-2s},&H>N.
 \end{cases}
\end{equation}

\begin{lemma}\label{lem:sums}
Let $\frac{1}{2}<s<1$, $p\ge0$, and $c,C>0$ be fixed. For $H\ge1$,
\begin{align}
 \sum_{K\lesssim H}\frac{\theta_{H,K}}{m_K}
 &\lesssim_s\frac{\ell_H}{m_H},\label{eq:sumlow}\\
 \sum_{T\gtrsim L}\frac{\ell_T^3}{T^3m_T^4}
 &\lesssim_s\frac{\ell_L^3}{L^3m_L^4},\qquad L\gg1,\label{eq:sumfour}\\
 \sum_{T\gtrsim H}\frac{\ell_T^2}{T^3m_T^3}
 &\lesssim_s\frac{\ell_H^2}{H^3m_H^3},\label{eq:sumthree}\\
 \sup_{H\ge cR}\delta_H\ell_H^p
 &\lesssim_{s,p,c}\frac{\ell_N^p}{R},\qquad 1\ll R\le N.\label{eq:sumdelta}
\end{align}
Also, for every dyadic $L>0$,
\begin{equation}\label{eq:sumLinfty}
 \sum_{K\le CL}\frac{\sqrt K}{m_K}\lesssim_{s,C}\frac{\sqrt L}{m_L}.
\end{equation}
\end{lemma}
\begin{proof}
For~\eqref{eq:sumlow}, split at $K_0=(C_s^2A^2\mu^2H)^{-1}$ and at $N$. The small-frequency gain makes the infinite tail below the first threshold summable; counting the remaining frequencies below the smoothing threshold gives the logarithm. Specifically,
\[
 \sum_{K\le K_0}\theta_{H,K}\lesssim1,\qquad
 \#\{K_0<K\le\min(H,N)\}\lesssim\ell_H,\qquad
 \sum_{N<K\lesssim H}m_K^{-1}\lesssim_s m_H^{-1}.
\]
For the next two sums, the weights above $N$ are
\[
 T^{-3}m_T^{-4}=N^{4s-4}T^{1-4s},\qquad
 T^{-3}m_T^{-3}=N^{3s-3}T^{-3s},
\]
whereas both are $T^{-3}$ below $N$. Since
$\ell_{2^jT}\le\ell_T+Cj$, the geometric-series bound
\[
 \sum_{j\ge0}(\ell_L+Cj)^p2^{-\alpha j}
 \lesssim_{\alpha,p}\ell_L^p,\qquad \alpha>0,
\]
proves \eqref{eq:sumfour} and \eqref{eq:sumthree}.
For \eqref{eq:sumdelta}, use $\delta_H=H^{-1}$ below $N$ and
\begin{equation}\label{eq:deltahigh}
 \delta_{2^jN}=N^{-1}2^{-j(2s-1)},\qquad
 \sup_{j\ge0}2^{-j(2s-1)}(\ell_N+Cj)^p\lesssim_{s,p}\ell_N^p.
\end{equation}
Finally, $K^{\frac{1}{2}}m_K^{-1}$ equals $K^{\frac{1}{2}}$ below $N$ and
$N^{s-1}K^{\frac{3}{2}-s}$ above $N$; summing either increasing geometric
sequence proves \eqref{eq:sumLinfty}.
\end{proof}

\subsection{Multilinear estimates}\label{sec:prelim-multilinear}
For general multilinear multiplier estimates, see Coifman--Meyer~\cite{CM78} and~\cite[Theorem~3.1]{CKSTT04}. Here we represent the localized multipliers by convolution kernels so that~\eqref{eq:pair} can be applied directly. At fixed convolution variables, Cauchy--Schwarz bounds a four-factor integral by two bilinear $L^2_{t,x}$ norms. The factors acquire different spatial shifts, which explains the relative-translation supremum in~\eqref{eq:Bdef}.

For time-independent finite complex measures,
\[
 \|\nu_1*\nu_2\|_{\TV}
 \le\|\nu_1\|_{\TV}\|\nu_2\|_{\TV}.
\]
If $T$ has an $L^1$ convolution kernel $K_T$, the kernel of $1-T$ is
the finite measure $\delta_0-K_T(x)\dd x$.

For time-independent finite measures $\nu_1,\nu_2$, Minkowski's
inequality gives
\begin{equation}\label{eq:convolutionpair}
 \begin{split}
 &\sup_y\|(\nu_1*u_H)(t,x+y)(\nu_2*u_K)(t,x)\|_{L^2_{t,x}}\\
 &\quad\le\|\nu_1\|_{\TV}\|\nu_2\|_{\TV}
       \sup_y\|u_H(t,x+y)u_K(t,x)\|_{L^2_{t,x}}.
 \end{split}
\end{equation}
By~\eqref{eq:convolutionpair} and~\eqref{eq:thetastability}, the estimate~\eqref{eq:pair} also holds for the fixed smooth cutoffs used below. We use $\BB(cR)$ to include neighboring dyadic frequencies. The fixed constant $0<c<1$ depends only on the cutoffs and frequency comparisons, and $cR\ge1$ is required. The next lemma treats a four-factor integral with an additional bounded weight.

\begin{lemma}\label{lem:kernel}
Let $f_i,g_i$ be measurable functions on $J\times\R^3$ such that
\[
 B_i=\sup_y\|f_i(t,x+y)g_i(t,x)\|_{L^2_{t,x}(J)}<\infty,
 \qquad i=1,2.
\]
For a time-independent finite complex measure $\nu$ on $(\R^3)^4$
and $q\in L^\infty(J\times\R^3)$, set
\[
 \mathcal L(t)=\int_{(\R^3)^4}\!\int_{\R^3}
 q(t,x)\prod_{i=1}^2
 [f_i(t,x-y_{2i-1})g_i(t,x-y_{2i})]\dd x\dd\nu(\vec y).
\]
Then
\begin{equation}\label{eq:kernelpair}
 \int_J|\mathcal L(t)|\dd t
 \le\|q\|_{L^\infty_{t,x}}\|\nu\|_{\TV}B_1B_2.
\end{equation}
The same estimate holds with any of the four factors conjugated.
\end{lemma}
\begin{proof}
Tonelli's theorem and Cauchy--Schwarz give
\begin{align*}
 \int_J|\mathcal L(t)|\dd t
 &\le\|q\|_{L^\infty_{t,x}}\int_{(\R^3)^4}
 \prod_{i=1}^2
 \|f_i(t,x-y_{2i-1})g_i(t,x-y_{2i})\|_{L^2_{t,x}}
 \,\dd|\nu|(\vec y)\\
 &\le\|q\|_{L^\infty_{t,x}}\|\nu\|_{\TV}B_1B_2,
\end{align*}
since the relative translation in the $i$th product is
$y_{2i}-y_{2i-1}$.
\end{proof}

For bounded measurable functions $f_1,\ldots,f_6$ and
time-independent finite complex measures $K_0,\ldots,K_6$, Fubini gives
\begin{align}
 &\prod_{j=1}^3(K_j*f_j)(x)
 \left[K_0*\left(\prod_{j=4}^6(K_j*f_j)\right)\right](x)\notag\\
 &\quad=\int\prod_{j=1}^3f_j(x-y_j)
               \prod_{j=4}^6f_j(x-z-y_j)
           \,\dd K_0(z)\prod_{j=1}^6\dd K_j(y_j)\notag\\
 &\quad=\int\prod_{j=1}^6 f_j(x-\eta_j)\dd\nu(\vec\eta),
 \qquad \|\nu\|_{\TV}\le\prod_{j=0}^6\|K_j\|_{\TV}.
 \label{eq:producttreekernel}
\end{align}
The final line uses $\eta_j=y_j$ for $j\le3$ and $\eta_j=z+y_j$ for $j\ge4$. Taking $K_j=\delta_0$ recovers $K_j*f_j=f_j$. Conjugation changes both a function and its convolution kernel but preserves total variation. The same representation applies to four factors when the outer convolution acts on three of them. The resulting measure $\nu$ need not be a product measure.

\section{Almost conservation of the modified energy}\label{sec:energy}
Since $I_NF(u)\ne F(I_Nu)$ in general, the modified energy~\eqref{eq:modifiedenergy} need not be conserved. We bound its variation using the finite constant $\BB$ from~\eqref{eq:Bdef}; Section~\ref{sec:recursion} later proves $\BB(c_EN)\lesssim_s1$.

Fix an absolute constant $c_E$ and write
\begin{equation}\label{eq:energythreshold}
 0<c_E\le2^{-16},\qquad B_*=\BB(c_EN).
\end{equation}
Throughout this section, assume \eqref{eq:boot} and \eqref{eq:roughA},
with $N\ge\max\{N_s,c_E^{-1}\}$. In particular, $B_*<\infty$ by
Lemma~\ref{lem:pair}.

\begin{proposition}[Almost conservation]\label{prop:energy}
Under the above assumptions,
\begin{equation}\label{eq:energyincrement}
 \int_J\left|\frac{\dd}{\dd t}E(I_Nu(t))\right|\dd t
 \lesssim_sB_*\left(\frac{\ell_N^3}{N}+\frac{\ell_N}{N^2}\right).
\end{equation}
\end{proposition}

We start from the modified-energy identity in~\cite[Section~4]{Dodson13}, write it in our Fourier convention, and estimate its quartic and sextic contributions separately.

Write $U=I_Nu$. For
$\Gamma_4=\{\xi_1-\xi_2+\xi_3-\xi_4=0\}$, define
\[
 \Lambda_4(M;u)=\int_{\Gamma_4}
 M(\vec\xi)\widehat u(\xi_1)\overline{\widehat u(\xi_2)}
             \widehat u(\xi_3)\overline{\widehat u(\xi_4)}
             \dd\xi_1\dd\xi_2\dd\xi_3,
 \quad \epsilon=(1,-1,1,-1).
\]

\begin{lemma}\label{lem:energyidentity}
The modified energy satisfies
\begin{equation}\label{eq:exactenergy}
 \frac{\dd}{\dd t}E(U)
 =\underbrace{\Imn\int\Delta\bar U\,[F(U)-I_NF(u)]}_{\mathcal E_4}
 +\underbrace{\Imn\int\overline{F(U)}I_NF(u)}_{\mathcal E_6}.
\end{equation}
With $\Phi_4=4\pi^2\sum_j\epsilon_j|\xi_j|^2$,
\begin{equation}\label{eq:quarticsymbol}
 \begin{aligned}
 \mathcal E_4&=-\frac{1}{4}\Imn\Lambda_4(\sigma_4;u),\\
 \sigma_4&=4\pi^2\sum_j\epsilon_j|\xi_j|^2m_N(\xi_j)^2
              -\Phi_4\prod_jm_N(\xi_j).
 \end{aligned}
\end{equation}
The spatial integral defining $\mathcal E_6$ has multiplier
\begin{equation}\label{eq:sexticsymbol}
 m_N(\xi_1-\xi_2+\xi_3)m_N(\xi_4)m_N(\xi_5)m_N(\xi_6),
\end{equation}
on $\xi_1-\xi_2+\xi_3=\xi_4-\xi_5+\xi_6$, with the Fourier
factors at positions $2,4,6$ conjugated.
\end{lemma}
\begin{proof}
Since $(\ii\partial_t+\Delta)U=I_NF(u)$,
\begin{align*}
 \frac{\dd}{\dd t}E(U)
 &=\Ren\int\overline{U_t}[-\Delta U+F(U)]\\
 &=\Ren\int\overline{U_t}[F(U)-I_NF(u)].
\end{align*}
Substituting $\overline{U_t}=-\ii\Delta\bar U+
\ii\overline{I_NF(u)}$ gives \eqref{eq:exactenergy}.
For the quartic term, write
\[
 a_j=4\pi^2\left[|\xi_j|^2m_N(\xi_j)^2
                 -|\xi_j|^2\prod_km_N(\xi_k)\right].
\]
Then $\mathcal E_4=\Imn\Lambda_4(a_4;u)$. Interchanging the two
unconjugated factors or the two conjugated factors, and taking
complex conjugates, gives
\[
 \Imn\Lambda_4(a_2;u)=\mathcal E_4,\qquad
 \Imn\Lambda_4(a_1;u)=\Imn\Lambda_4(a_3;u)=-\mathcal E_4,
\]
which proves \eqref{eq:quarticsymbol}. Expanding the two cubic factors
gives \eqref{eq:sexticsymbol}.
\end{proof}

The quartic term explains the normalization of~\eqref{eq:Bdef}. A nonzero dyadic contribution has frequencies $H\sim H'\ge M\ge K$, with $H'\ge c_EN$, and a convolution kernel of total variation at most $C_sH^2m_H^2$. Pairing $(H,M)$ and $(H',K)$ therefore gives
\begin{equation}\label{eq:quarticpairbound}
 \int_J|\mathcal E_{4;\vec H}(t)|\dd t
 \lesssim_s H^2m_H^2
 \frac{B_*\ell_H}{H^3m_H^2m_Mm_K}
 \theta_{H,M}\theta_{H,K}.
\end{equation}
Here $\vec H=(H,H',M,K)$ lists frequencies in decreasing order. For each ordering of the original frequencies, $\mathcal E_{4;\vec H}$ denotes the corresponding term with its original conjugations. There are only finitely many orderings, which we estimate separately; the sextic terms use the same convention. The factor $H^{-1}$ in the quartic bound becomes $N^{-1}$ after summation. For the sextic term, we place two factors in $L^\infty_{t,x}$ and retain a multiplier gain $m_L^2$, with $L$ the fourth largest frequency. We prove these kernel bounds before summing the dyadic contributions.

\subsection{Kernel bounds and frequency separation}
\begin{lemma}\label{lem:Ikernel}
For fixed $\frac{1}{2}<s<1$ and $c>0$, $N\ge1$, and dyadic $H,L>0$,
\begin{equation}\label{eq:Ikernel}
 \|K_{I_N}\|_1\lesssim_s1,\qquad
 \|K_{I_NP_{\ge cL}}\|_1\lesssim_{s,c}m_L,\qquad
 \|K_{I_NP_H}\|_1\lesssim_sm_H .
\end{equation}
\end{lemma}
\begin{proof}
Rescaling on the support of $P_H$ gives
\[
 |K_{I_NP_H}(x)|\lesssim_s m_HH^3(1+H|x|)^{-4}.
\]
The kernel of $I_NP_{\le N}$ has uniformly bounded $L^1$ norm, and
the norms of the kernels of $I_NP_H$ are summable for $H>N$, since
$m_{2^jN}=2^{-j(1-s)}$ for $j\ge1$. This proves the first and third
bounds. If $L>N$, the same sum over $H\gtrsim cL$ is
$O_{s,c}(m_L)$. If $L\le N$, use the first bound and the uniformly
bounded total variation of the kernel of $P_{\ge cL}$.
The same argument applies to the fixed smooth high-frequency
cutoffs used below.
\end{proof}

\begin{lemma}\label{lem:energythreshold}
Let $u_{\rm lo}=P_{\le N/100}u$ and $u_{\rm hi}=u-u_{\rm lo}$.
Every nonzero dyadic contribution to $\mathcal E_4(u)$ or to the
multilinear expansion of $\mathcal E_6(u)-\mathcal E_6(u_{\rm lo})$
in $u_{\rm lo},u_{\rm hi}$ has two comparable largest frequencies,
both at least $c_EN$.
\end{lemma}
\begin{proof}
For a quartic or sextic interaction, write $\epsilon_j=(-1)^{j+1}$
for the signs in its frequency relation. If $T,T'$ are the two largest
frequencies, then
\[
 T'<T/20\ \Longrightarrow\
 |\xi_{\max}|\ge T/2>\sum_{j\ne\max}|\xi_j|,
 \qquad
 \sum_j\epsilon_j\xi_j=0\ \Longrightarrow\ T'\ge T/20 .
\]
Thus
\[
 \sigma_4\ne0\ \Longrightarrow\
 \max_j|\xi_j|>N\ \Longrightarrow\ T\ge N/2,\quad T'\ge N/40 .
\]
The Fourier supports of $u_{\rm lo}$ and $u_{\rm hi}$ give
\begin{gather*}
 I_Nu_{\rm lo}=u_{\rm lo},\qquad
 I_NF(u_{\rm lo})=F(u_{\rm lo}),\qquad
 \mathcal E_6(u_{\rm lo})=\Imn\int|u_{\rm lo}|^6=0,\\
 \supp\widehat{u_{\rm hi}}\subset\{|\xi|\ge N/100\},\qquad
 T\ge N/200,\qquad T'\ge N/4000>c_EN .
\qedhere
\end{gather*}
\end{proof}

Subtract the completely low-frequency sextic contribution before taking absolute values. Every remaining monomial then contains a high-frequency factor, so the preceding frequency-separation lemma applies. The remaining expansion has finitely many terms with factors $P_{\le N/100}u_H$ or $(1-P_{\le N/100})u_H$. These cutoffs preserve the original Fourier supports and have uniformly bounded measure kernels, so~\eqref{eq:convolutionpair} remains applicable.

\subsection{Quartic variation}
\begin{lemma}\label{lem:quartickernel}
On a nonzero quartic term, let $H=\max_jH_j$ and set
$\chi_{H_j}(\xi)=\widetilde\chi(\xi/H_j)$, where
$\widetilde\chi$ is smooth, real and radial, supported in
$\frac{1}{4}\le|\xi|\le4$ and equal to one for $\frac{1}{2}\le|\xi|\le2$.
The quadrilinear form with symbol $\sigma_4\prod_j\chi_{H_j}$,
evaluated on $u_{H_1},\ldots,u_{H_4}$, can be written as
\[
 \int_{(\R^3)^4}\!\int_{\R^3}
 u_{H_1}(x-y_1)\overline{u_{H_2}(x-y_2)}
 u_{H_3}(x-y_3)\overline{u_{H_4}(x-y_4)}
 \,\dd x\dd\nu_4(\vec y),
 \qquad \|\nu_4\|_{\TV}\lesssim_s H^2m_H^2,
\]
where $\nu_4$ is independent of time.
\end{lemma}
\begin{proof}
Let $H\sim H'\ge M\ge K$ be the dyadic frequencies in decreasing order.
For a symbol $b$, write $K_b=\check b$ for the convolution kernel of $b(D)$.
By \eqref{eq:smoothkernel},
\begin{align*}
 \|K_{|\xi|^2m_N^2\chi_{H_j}}\|_1
 &\lesssim_s H_j^2m_{H_j}^2,\\
 \|K_{|\xi|^2m_N\chi_{H_j}}\|_1
 &\lesssim_s H_j^2m_{H_j},\qquad
 \|K_{m_N\chi_{H_j}}\|_1\lesssim_s m_{H_j}.
\end{align*}
We estimate the kernels of the two terms in
$\sigma_4\prod_k\chi_{H_k}$ separately:
\begin{align*}
 \sum_j\|K_{|\xi|^2m_N^2\chi_{H_j}}\|_1
          \prod_{k\ne j}\|K_{\chi_{H_k}}\|_1
 &\lesssim_s\sum_jH_j^2m_{H_j}^2
 \lesssim_sH^2m_H^2,\\
 \sum_j\|K_{|\xi|^2m_N\chi_{H_j}}\|_1
          \prod_{k\ne j}\|K_{m_N\chi_{H_k}}\|_1
 &\lesssim_s\sum_jH_j^2\prod_{k=1}^4m_{H_k}\\
 &\lesssim_sH^2m_Hm_{H'}m_Mm_K
 \lesssim_sH^2m_H^2.
 \qedhere
\end{align*}
\end{proof}

\begin{proposition}\label{prop:quarticvariation}
The quartic term in \eqref{eq:exactenergy} satisfies
\begin{equation}\label{eq:E4bound}
 \int_J|\mathcal E_4(t)|\dd t\lesssim_sB_*\frac{\ell_N^3}{N}.
\end{equation}
\end{proposition}
\begin{proof}
Lemma~\ref{lem:energythreshold} gives $H'\sim H$ and $H\gtrsim N$. Apply Lemmas~\ref{lem:kernel} and~\ref{lem:quartickernel} to the pairs $(H,M)$ and $(H',K)$, obtaining~\eqref{eq:quarticpairbound}. There are only boundedly many choices of $H'$ for each $H$. The paired bound contributes one logarithm, and each lower-frequency sum in~\eqref{eq:sumlow} contributes another:
\begin{align*}
 \int_J|\mathcal E_4|
 &\lesssim_sB_*\sum_{H\gtrsim N}\frac{\ell_H}{H}
        \left(\sum_{M\lesssim H}\frac{\theta_{H,M}}{m_M}\right)^2\\
 &\lesssim_sB_*\sum_{H\gtrsim N}\frac{\ell_H^3}{Hm_H^2}
 \lesssim_s\frac{B_*}N
       \sum_{j\ge0}(\ell_N+Cj)^3\,2^{-j(2s-1)}
 \lesssim_sB_*\frac{\ell_N^3}N .
\end{align*}
Above the smoothing scale,~\eqref{eq:deltahigh} turns the normalized weight into a geometric sequence. Its decay requires $2s-1>0$ and dominates the logarithmic growth. The finitely many frequencies with $H\gtrsim N$ and $H<N$ are covered by~\eqref{eq:thetastability}.
\end{proof}

\subsection{Sextic variation}
Order the six dyadic frequencies as
\[
 H\ge H'\ge M\ge L\ge K\ge J_{\min},\qquad H'\sim H.
\]
Here $L$ is the fourth largest dyadic frequency. The factors from $\overline{F(I_Nu)}$ occupy positions $4,5,6$ in~\eqref{eq:sexticsymbol}. Retaining two factors of $m_L$ is what makes the eventual sum converge for $s>\frac{1}{2}$; the next lemma supplies this gain.

\begin{lemma}\label{lem:compressed}
Each sextic spatial integral in \eqref{eq:exactenergy}, localized at
frequencies $H_1,\ldots,H_6$, can be written as
\[
 \int_{(\R^3)^6}\!\int_{\R^3}
 \prod_{j=1}^3
 u_{H_{2j-1}}(x-y_{2j-1})\overline{u_{H_{2j}}(x-y_{2j})}
 \,\dd x\dd\nu_{\vec H}(\vec y),\qquad
 \|\nu_{\vec H}\|_{\TV}\lesssim_s m_L^2,
\]
where $L$ is the fourth largest frequency and $\nu_{\vec H}$ is
independent of time.
\end{lemma}
\begin{proof}
The smoothed cubic factor uses the original Fourier positions~$4,5,6$. Write its product and the other cubic product as
\[
 V=(I_Nu_{H_4})\overline{I_Nu_{H_5}}(I_Nu_{H_6}),\qquad
 W=u_{H_1}\overline{u_{H_2}}u_{H_3}.
\]

Order the three frequencies in $F(I_Nu)$ as $S_1\ge S_2\ge S_3$. We obtain $m_L^2$ either from two individual factors or from one such factor and the multiplier $I_N$ acting on $F(u)$.
For each factor, use
$\chi_{S_i}(\xi)=\widetilde\chi(\xi/S_i)$ with the fixed cutoff
$\widetilde\chi$ from Lemma~\ref{lem:quartickernel}. Thus
$I_Nu_{S_i}=I_N\chi_{S_i}(D)u_{S_i}$, and
$\|K_{I_N\chi_{S_i}(D)}\|_1\lesssim_s m_{S_i}$ by
\eqref{eq:smoothkernel}.
\emph{Step 1: Two large frequencies in the smoothed cubic.} If $S_2\ge L/64$, two individual smoothing factors supply the required gain. Lemma~\ref{lem:Ikernel} and~\eqref{eq:producttreekernel} give
\[
 \|\nu\|_{\TV}
 \lesssim_s\|K_{I_N}\|_1
             \prod_{i=1}^3\|K_{I_N\chi_{S_i}(D)}\|_1
 \lesssim_s m_{S_1}m_{S_2}m_{S_3}\lesssim_s m_L^2.
\]
\emph{Step 2: One large frequency and a high-frequency output.} If $S_2<L/64$, at least four of all six frequencies are at least $L$. Only three lie in the other cubic factor, so the largest frequency of the smoothed cubic must reach $L$. The two smaller frequencies cannot cancel the largest one:
\[
 S_1\ge L,\qquad S_2,S_3<L/64,\qquad
 |\xi_4-\xi_5+\xi_6|\ge L/2-2L/64-2L/64=7L/16.
\]
Choose $\eta_L(\xi)=\eta(\xi/L)$, with a fixed smooth real radial
cutoff $\eta$, such that
\[
 \eta_L(\xi)=0\ (|\xi|\le L/4),\qquad
 \eta_L(\xi)=1\ (|\xi|\ge7L/16).
\]
The Fourier support of $V$ gives $\eta_L(D)V=V$. By self-adjointness, we may transfer $\eta_L(D)$ to $I_NW$. Lemma~\ref{lem:Ikernel} then supplies the second smoothing gain, in addition to the one from the largest individual frequency:
\begin{align*}
 \int\bar V I_NW&=\int\bar V I_N\eta_L(D)W,\\
 \|\nu\|_{\TV}
 &\lesssim_s\|K_{I_N\eta_L(D)}\|_1
               \prod_{i=1}^3\|K_{I_N\chi_{S_i}(D)}\|_1
 \lesssim_s m_Lm_{S_1}m_{S_2}m_{S_3}\lesssim_s m_L^2.
\end{align*}
Writing $I_NW$ or $I_N\eta_L(D)W$ as a convolution shifts all three factors of $W$ by a common vector. Formula~\eqref{eq:producttreekernel} then gives the required measure representation. The additional cutoffs $P_{\le N/100}$ and $1-P_{\le N/100}$ from Lemma~\ref{lem:energythreshold} have uniformly bounded measure kernels and preserve the bound.
\end{proof}

\begin{proposition}\label{prop:sexticvariation}
The sextic term in \eqref{eq:exactenergy} satisfies
\begin{equation}\label{eq:E6bound}
 \int_J|\mathcal E_6(t)|\dd t\lesssim_sB_*\frac{\ell_N}{N^2}.
\end{equation}
\end{proposition}
\begin{proof}
Subtract $\mathcal E_6(u_{\rm lo})=0$, and apply Lemma~\ref{lem:compressed}. At fixed kernel translations, place the two lowest frequencies $K,J_{\min}$ in $L^\infty_{t,x}$ using~\eqref{eq:massfreq}. Pair the remaining factors at $(H,M)$ and $(H',L)$ in $L^2_{t,x}$, then integrate against $|\nu_{\vec H}|$ using the translation suprema as in Lemma~\ref{lem:kernel}. Denote this dyadic contribution by $\mathcal E_{6;\vec H}(t)$, where $\vec H=(H,H',M,L,K,J_{\min})$. The resulting bound is
\begin{equation}\label{eq:sixblock}
 \int_J|\mathcal E_{6;\vec H}(t)|\dd t
 \lesssim_s
 \frac{B_*\ell_H}{H^3m_H^2m_Mm_L}
       \frac{\sqrt{KJ_{\min}}}{m_Km_{J_{\min}}}\ m_L^2
 =
 \frac{B_*\ell_H}{H^3m_H^2}
       \frac{\sqrt{KJ_{\min}}\,m_L}{m_Mm_Km_{J_{\min}}}.
\end{equation}
In~\eqref{eq:sixblock}, the factors $\theta$ from~\eqref{eq:pair} have already been bounded by one. The two bilinear bounds contribute one logarithm. The positive powers of the two lowest frequencies from the $L^\infty$ estimates make the remaining sums geometric:
\begin{equation}\label{eq:sixsums}
 \sum_{J_{\min}\le K\le L}\frac{\sqrt{KJ_{\min}}}{m_Km_{J_{\min}}}
 \lesssim_s\frac L{m_L^2},\qquad
 \sum_{L\le M}\frac L{m_L}\lesssim_s\frac M{m_M},\qquad
 \sum_{M\le H}\frac M{m_M^2}\lesssim_s\frac H{m_H^2}.
\end{equation}
Sum~\eqref{eq:sixblock} first in $J_{\min},K$, then in $L$, and finally in $M$ to obtain
\[
 \int_J|\mathcal E_6|
 \lesssim_sB_*\sum_{H\gtrsim N}\frac{\ell_H}{H^2m_H^4}
 \lesssim_s\frac{B_*}{N^2}
       \sum_{j\ge0}(\ell_N+Cj)2^{-j(4s-2)}
 \lesssim_sB_*\frac{\ell_N}{N^2}.
\]
The exponent $4s-2$ is positive throughout the stated regularity range, so the series converges. As in the quartic estimate, the finitely many dyadic frequencies satisfying $H\gtrsim N$ and $H<N$ are covered by~\eqref{eq:thetastability}.
\end{proof}

\begin{proof}[Proof of Proposition~\ref{prop:energy}]
Combine \eqref{eq:exactenergy}, \eqref{eq:E4bound}, and
\eqref{eq:E6bound}.
\end{proof}

The remaining task is to prove $\BB(c_EN)\lesssim_s1$.

\section{Directional bilinear interaction estimate}\label{sec:flux}
For separated frequencies, we pair high-frequency mass with low-frequency energy. Two one-dimensional models motivate the forced interaction estimate derived below.

The gain from frequency separation can first be seen in a transport
model. Let smooth, rapidly decreasing nonnegative densities on $\R$
satisfy
\[
 \rho_t+v_H\rho_x=S_H,\qquad e_t+v_Le_x=S_L,\qquad v_H>v_L.
\]
Set $q_e^\eta(x)=\int_{y<x-\eta}e(y)\dd y$ and
$q_\rho^\eta(y)=\int_{x>y+\eta}\rho(x)\dd x$. Integration by parts gives
\begin{equation}\label{eq:transportcollision}
 \begin{aligned}
 \mathcal G_\eta(t)&=\iint_{x-y>\eta}\rho(t,x)e(t,y)\dd x\dd y,\\
 \mathcal G_\eta'(t)&=(v_H-v_L)\int_\R\rho(t,y+\eta)e(t,y)\dd y
             +\int q_e^\eta S_H+\int q_\rho^\eta S_L.
 \end{aligned}
\end{equation}
Writing $M=\sup_J\|\rho(t)\|_1$ and $E=\sup_J\|e(t)\|_1$, we have
$0\le\mathcal G_\eta\le ME$. Hence
\[
 (v_H-v_L)\int_J\!\int_\R\rho(t,y+\eta)e(t,y)\dd y\dd t
 \le ME+E\|S_H\|_{L^1(J\times\R)}+
       \int_J\left|\int q_\rho^\eta S_L\right|\dd t.
\]
This has the structure of Proposition~\ref{prop:flux}: when $v_H-v_L\sim H$, integration gains a factor $H^{-1}$. We retain the weight on $S_L$ because its Fourier support will be used in the nonlinear estimate. The shift $\eta$ is fixed in time. Allowing it to vary would replace $v_H-v_L$ in~\eqref{eq:transportcollision} by $v_H-v_L-\eta'(t)$ and could destroy positivity.

For the Schr\"odinger equation, the corresponding positivity follows
from Plancherel's identity. Let $f,g$ be Schwartz solutions of
$(\ii\partial_t+\partial_x^2)f=(\ii\partial_t+\partial_x^2)g=0$ on $\R$,
with Fourier supports in $[\gamma H,CH]$ and $[-CK,CK]$, respectively.
Take $\rho=|f|^2$ and $e=|g|^2$ in the definition of
$\mathcal G_\eta$, and set $W_\eta(y)=f(y+\eta)\overline{g(y)}$.
The mass identities and Plancherel give
\begin{equation}\label{eq:modelFourierpositive}
 \mathcal G_\eta'=2\Imn\int\overline{W_\eta}\,\partial_yW_\eta
 =4\pi\int\xi|\widehat{W_\eta}(\xi)|^2\dd\xi
 \ge4\pi(\gamma H-CK)\|W_\eta\|_2^2.
\end{equation}
Here $\supp\widehat{W_\eta}\subset[\gamma H-CK,CH+CK]$. Since
$0\le\mathcal G_\eta\le\|f(0)\|_2^2\|g(0)\|_2^2$, integration gives
\[
 \int_J\|f(t,\cdot+\eta)g(t)\|_2^2\dd t
 \lesssim H^{-1}\|f(0)\|_2^2\|g(0)\|_2^2,\qquad K\ll H.
\]
In the three-dimensional argument, we use the same one-dimensional Fourier positivity after fixing two independent transverse variables for the components of the low-frequency energy.
Physical-space approaches to bilinear estimates include the bilinear
virial identities of Planchon and Vega \cite[Section~2]{PV}
and the local-smoothing
argument of Tao \cite{TaoBilinear}; see also the density--flux identities
in the one-dimensional work of Ifrim and Tataru \cite{IT23}.

\subsection{Smoothed densities and trace estimates}
The nonlinear estimate in Section~\ref{sec:exterior} uses a weight with Fourier support much narrower than $L$ to exclude interactions with only one large frequency. Smoothing $|h|^2$ longitudinally supplies that localization. A nonnegative kernel of integral one also preserves the density's positivity and total mass. We require Fourier support in $[-2r,2r]$ and work with smooth, rapidly decreasing fields on a compact interval $J$.
Fix a unit vector $\omega$ and a real even $\psi\in\mathcal S(\R)$ with
$\|\psi\|_2=1$, $\psi(0)\ne0$, and
$\supp\widehat\psi\subset[-1,1]$. For $r>0$, set
\begin{equation}\label{eq:densitykernel}
 \begin{gathered}
 k_r(z)=r|\psi(rz)|^2,\qquad
 (k_r*_\omega f)(x)=\int_\R k_r(z)f(x-z\omega)\dd z,\\
 k_r\ge0,\qquad k_r(-z)=k_r(z),\qquad
 \int_\R k_r=1,\qquad
 \supp\widehat{k_r}\subset[-2r,2r].
 \end{gathered}
\end{equation}
The direction $\omega$ and parameter $r$ remain fixed in time.
For $(\ii\partial_t+\Delta)h=G$, define
\begin{equation}\label{eq:smootheddensity}
 \rho_h=k_r*_\omega|h|^2,\qquad
 J_h=k_r*_\omega\Imn(\bar h\nabla h),\qquad
 s_h=2k_r*_\omega\Imn(\bar hG).
\end{equation}

Convolving the mass identity
\begin{equation}\label{eq:canonicalmass}
 \partial_t|h|^2+2\nabla\cdot\Imn(\bar h\nabla h)
 =2\Imn(\bar hG)
\end{equation}
with $k_r$ and using $\|k_r\|_1=1$, we obtain
\begin{equation}\label{eq:densitymass}
 \partial_t\rho_h+2\nabla\cdot J_h=s_h,\qquad
 \int\rho_h=\|h\|_2^2,\qquad
 \|s_h\|_{L^1_{t,x}}\le2\|hG\|_{L^1_{t,x}}.
\end{equation}

We must recover the original product after smoothing the high-frequency density. Apply the pointwise estimate to the low-frequency factor and transfer its convolution to the high density by evenness of the kernel; the loss is therefore determined by the low frequency. The transverse trace estimate then bounds the low factor on each plane by its energy, at a logarithmic cost from two-dimensional Fourier inversion.
\begin{lemma}\label{lem:positive}\label{lem:trace}
If $\supp\widehat b(t)\subset B(0,CK)$ for every $t\in J$, then
\begin{equation}\label{eq:productpositive}
 |b|^2\lesssim(1+K/r)(k_r*_\omega|b|^2),\qquad
 \|hb\|_{L^2_{t,x}}^2
 \lesssim(1+K/r)\int_J\!\int\rho_h|b|^2.
\end{equation}
If $\supp\widehat w(t)\subset B(0,2L)$ for every $t\in J$ and $a>0$, then
\begin{equation}\label{eq:transversetrace}
 \sup_{y\in\omega^\perp}|w(\rho\omega+y)|^2
 \lesssim\log(2+L/a)
       \int_{\omega^\perp}(a^2|w|^2+|\nabla_\perp w|^2)(\rho\omega+y)\dd y.
\end{equation}
\end{lemma}
\begin{proof}
For fixed $x$, the function
$B_x(z)=r^{\frac{1}{2}}\psi(rz)b(x-z\omega)$ has Fourier support in
$[-r-CK,r+CK]$. The one-dimensional Bernstein inequality gives
\[
 r|\psi(0)|^2|b(x)|^2=|B_x(0)|^2
 \lesssim(K+r)\|B_x\|_2^2
 =(K+r)(k_r*_\omega|b|^2)(x).
\]
The second inequality in \eqref{eq:productpositive} follows by Fubini
and the evenness of $k_r$. For \eqref{eq:transversetrace}, apply
Cauchy--Schwarz to transverse Fourier inversion with weight
$a^2+|\zeta|^2$, using
\[
 \int_{|\zeta|\le2L}\frac{\dd\zeta}{a^2+|\zeta|^2}
 =\pi\log(1+4L^2/a^2)\lesssim\log(2+L/a).
 \qedhere
\]
\end{proof}

\subsection{The low-frequency energy identity}
Let $P=P_{\le L}=\chi(D/L)$ and consider
\begin{equation}\label{eq:abstractforced}
 (\ii\partial_t+\Delta)h=G,\qquad
 (\ii\partial_t+\Delta)w=PF(v)+R,\qquad v=Pw,
\end{equation}
where $\gamma,C_h>0$ are fixed and the following spatial Fourier support
conditions hold for every $t\in J$:
\begin{equation}\label{eq:abstractsupports}
 \supp\widehat h\subset\{\xi\cdot\omega\ge\gamma H,\ |\xi|\le C_hH\},
 \qquad
 \supp\widehat w,\supp\widehat R\subset B(0,2L),\qquad r>0.
\end{equation}
To estimate products with $h$, introduce an auxiliary low-frequency energy density. For $a>0$, set
\begin{equation}\label{eq:abstractenergy}
\begin{aligned}
 e&=a^2|w|^2+|\nabla w|^2+\frac{1}{2}|v|^4,\\
 j_w&=a^2\Imn(\bar w\nabla w)
 +\sum_{j=1}^3\Imn(\partial_j\bar w\,\nabla\partial_jw)
 +|v|^2\Imn(\bar v\nabla v),\\
 E_w&=\sup_J\int e,\qquad M_h=\sup_J\|h\|_2^2.
\end{aligned}
\end{equation}
Moving $P$ across the weighted kinetic term produces derivatives of $v=Pw$ and a commutator. With the potential $\frac{1}{2}|v|^4$, the local nonlinear contributions from the kinetic and potential terms cancel. The remaining terms are the transport terms, commutators, and the contribution of $R$. The modified energy remains $E(I_Nu)$.

For $w=P_{\le L}u$ under~\eqref{eq:boot}, choose $a=\mu^{-1}$, so that $a^2\|w\|_2^2\lesssim1$. Together with the gradient and potential estimates, this gives a low-frequency energy bound without a positive power of $\mu$; see Lemma~\ref{lem:actualEL}.

For the positivity argument, write
\begin{equation}\label{eq:energyvector}
 \begin{gathered}
 Z_0=aw,\qquad Z_j=\partial_jw\ (1\le j\le3),\qquad Z_4=v^2/\sqrt2,\\
 e=\sum_{\alpha=0}^4|Z_\alpha|^2,\qquad
 j_w=\sum_{\alpha=0}^4\Imn(\bar Z_\alpha\nabla Z_\alpha).
 \end{gathered}
\end{equation}
Each $Z_\alpha$ has Fourier support in $B(0,4L)$.
We write $[A,B]=AB-BA$, with scalar weights acting by multiplication.

\begin{lemma}\label{lem:lowbalance}
For each fixed time and every bounded real smooth weight $q=q(t,x)$,
\begin{equation}\label{eq:lowbalance}
 \int q(e_t+2\nabla\cdot j_w)=\mathcal I_q(w)+\DD_q(w,R),
\end{equation}
where
\begin{align}
 \mathcal I_q(w)
 &=2a^2\Imn\int[P,q]\bar wF(v)
 +2\Imn\sum_j\int[P,q]\partial_j\bar w\,\partial_jF(v)
 +\frac1\ii\ip{F(v)}{[q,P^2]F(v)},\label{eq:internalformula}\\
 \DD_q(w,R)
 &=2\Imn\int q\left(a^2\bar wR+\nabla\bar w\cdot\nabla R+
                       \overline{F(v)}PR\right).\label{eq:abstractwork}
\end{align}
\end{lemma}
\begin{proof}
Applying $P$ to \eqref{eq:abstractforced} gives
$(\ii\partial_t+\Delta)v=P^2F(v)+PR$. Write
$\varrho=|v|^2$ and $j_v=\Imn(\bar v\nabla v)$. In the weighted
mass and gradient identities, self-adjointness of $P$ gives
\[
 P(q\bar w)=q\bar v+[P,q]\bar w,\qquad
 P(q\partial_j\bar w)=q\partial_j\bar v+[P,q]\partial_j\bar w.
\]
The terms without commutators satisfy
\begin{align*}
 \Imn(\bar vF(v))&=0,\qquad
 2\Imn(\nabla\bar v\cdot\nabla F(v))=-2\nabla\varrho\cdot j_v,\\
 \partial_t(\frac{1}{2}\varrho^2)+2\nabla\cdot(\varrho j_v)
 &=2\nabla\varrho\cdot j_v+
 2\Imn\!\left(\overline{F(v)}[P^2F(v)+PR]\right).
\end{align*}
Thus the local terms from the gradient and potential energies cancel.
Finally,
\[
 2\Imn\ip{F(v)}{qP^2F(v)}
 =\ii^{-1}\ip{F(v)}{[q,P^2]F(v)},
\]
which yields \eqref{eq:lowbalance}.
\end{proof}

\subsection{The half-space interaction}
We integrate the density in \eqref{eq:smootheddensity} over the
transverse variables and set
\begin{equation}\label{eq:marginals}
 n_r(t,\rho)=\int_{\omega^\perp}\rho_h(t,\rho\omega+y)\dd y
            =k_r*n_h(t,\rho),\qquad
 n_h(t,\rho)=\int_{\omega^\perp}|h(t,\rho\omega+y)|^2\dd y.
\end{equation}
For fixed $\eta\in\R$, define
\begin{equation}\label{eq:halfweights}
 q_r^\eta(t,y)=\int_{y\cdot\omega+\eta}^{\infty}n_r(t,\rho)\dd\rho,\qquad
 q_e^\eta(t,x)=\int_{y\cdot\omega<x\cdot\omega-\eta}e(t,y)\dd y.
\end{equation}
Then
\begin{equation}\label{eq:weightbandwidth}
 0\le q_r^\eta\le M_h,\quad 0\le q_e^\eta\le E_w,\quad
 \|n_r\|_\infty\lesssim rM_h,\quad
 \supp\widehat{q_r^\eta}\subset\{\zeta\omega:|\zeta|\le2r\}.
\end{equation}
We also write $q_r^\eta(t,\rho)$ when $\rho=x\cdot\omega$.
Then $\partial_\rho q_r^\eta(t,\rho)=-n_r(t,\rho+\eta)$ has Fourier
support in $[-2r,2r]$. Since the weight is independent of the transverse
variables, its spatial Fourier support is the segment in
\eqref{eq:weightbandwidth}, in the sense of distributions. Moreover,
$\|\partial_\rho^j q_r^\eta\|_\infty\le C_{\psi,j}r^jM_h$
for $j\ge1$.

Set
\begin{align}
 \mathcal Q_\eta(t)
 &=\iint_{(x-y)\cdot\omega>\eta}\rho_h(t,x)e(t,y)\dd x\dd y,
       &0\le\mathcal Q_\eta\le M_hE_w,\label{eq:relativeaction}\\
 C_\eta(t)
 &=\int_\R n_r(t,\rho+\eta)
       \int_{\omega^\perp}e(t,\rho\omega+y)\dd y\dd\rho,
       &\CC=\sup_\eta\int_J C_\eta(t)\dd t<\infty.\label{eq:fluxC}
\end{align}
Indeed, $\CC\lesssim rM_hE_w|J|$.

\begin{proposition}[Forced directional interaction estimate]\label{prop:flux}
Under \eqref{eq:abstractforced}--\eqref{eq:abstractsupports}, assume
\begin{equation}\label{eq:fluxseparation}
 H\ge C_{\chi,\gamma}(L+E_w).
\end{equation}
Then
\begin{equation}\label{eq:forcedflux}
 H\CC\lesssim_{\chi,\gamma}
 M_hE_w+E_w\|hG\|_{L^1_{t,x}}
 +\sup_\eta\int_J|\DD_{q_r^\eta}(w,R)|\dd t.
\end{equation}
For a time-independent Fourier multiplier $T_K$ whose symbol is supported in $B(0,CK)$
and whose kernel satisfies $\|K_{T_K}\|_1\le C_T$ uniformly in $K$,
where $K\le L/2$,
\begin{equation}\label{eq:fluxproduct}
 \|h(t,x+x_0)T_Kw(t,x)\|_{L^2_{t,x}}^2
 \lesssim_{C,C_T}(1+K/r)\log(2+L/a)\CC,\qquad x_0\in\R^3.
\end{equation}
\end{proposition}

The Fourier calculation in~\eqref{eq:modelFourierpositive} makes the principal term positive. We establish this lower bound first, then estimate the commutators produced by the low-frequency energy identity.

\begin{lemma}\label{lem:relative}
We have
\begin{equation}\label{eq:Qprime}
 \mathcal Q_\eta'
 =\mathcal P_\eta+\int q_e^\eta s_h
             +\mathcal I_{q_r^\eta}(w)+\DD_{q_r^\eta}(w,R),
\end{equation}
where
\begin{equation}\label{eq:principalpositive}
 \mathcal P_\eta
 =2\iint\delta((x-y)\cdot\omega-\eta)
       [J_h(x)\cdot\omega\,e(y)-\rho_h(x)j_w(y)\cdot\omega]
 \ge4\pi(\gamma H-4L)C_\eta(t).
\end{equation}
\end{lemma}
\begin{proof}
\emph{Step 1: Differentiate the interaction.} Write $\mathcal Q_\eta=\int q_r^\eta e$. The smoothed mass identity~\eqref{eq:densitymass} gives
\[
 \partial_tq_r^\eta(t,y)
 =2\int_{\omega^\perp}
 J_h(t,(y\cdot\omega+\eta)\omega+z)\cdot\omega\dd z
 +\int_{(x-y)\cdot\omega>\eta}s_h(t,x)\dd x.
\]
Use \eqref{eq:lowbalance} for the other derivative and integrate by
parts. The two current terms give $\mathcal P_\eta$, and Fubini gives
the forcing term $\int q_e^\eta s_h$, proving \eqref{eq:Qprime}.

\emph{Step 2: Positivity.} For the lower bound in~\eqref{eq:principalpositive}, fix $t,z$ and set $\Omega=\R\times\omega^\perp\times\omega^\perp$. The half-space boundary couples only the longitudinal coordinates, so the two transverse positions remain independent. The smoothing variable translates the high-frequency field longitudinally; each fixed translate has the same Fourier-support lower bound. Define
\begin{equation}\label{eq:directproduct}
 W_{\alpha,z}(\rho,x_\perp,y_\perp)
 =h(t,(\rho+\eta-z)\omega+x_\perp)
       \overline{Z_\alpha(t,\rho\omega+y_\perp)}.
\end{equation}
The partial Fourier transform in $\rho$ satisfies
\begin{equation}\label{eq:directsupport}
 \supp\mathcal F_\rho W_{\alpha,z}
 \subset[\gamma H-4L,C_hH+4L].
\end{equation}
Applying the one-dimensional Bernstein inequality in $\rho$,
with values in $L^2(\omega^\perp)$, gives
\[
 \sup_\rho\sum_\alpha\int_{\omega^\perp}|Z_\alpha(\rho,y)|^2\dd y
 \lesssim LE_w,
 \qquad
 \sum_\alpha\|W_{\alpha,z}\|_{L^2(\Omega)}^2\lesssim LM_hE_w.
\]
Together with \eqref{eq:directsupport}, this justifies the following
integrations against $k_r(z)\dd z$. By \eqref{eq:energyvector} and
Plancherel in $\rho$,
\begin{align*}
 C_\eta(t)&=\sum_{\alpha=0}^4\int_\R
           k_r(z)\|W_{\alpha,z}\|_{L^2(\Omega)}^2\dd z,\\
 \frac{1}{2}\mathcal P_\eta
 &=\sum_{\alpha=0}^4\int_\R k_r(z)
     \Imn\int_\Omega\overline{W_{\alpha,z}}\partial_\rho W_{\alpha,z}\dd\Omega\dd z\\
 &=\sum_{\alpha=0}^4\int_\R k_r(z)
     \int_{\R\times\omega^\perp\times\omega^\perp}
       2\pi\xi|\mathcal F_\rho W_{\alpha,z}|^2\dd\xi\dd x_\perp\dd y_\perp\dd z\\
 &\ge2\pi(\gamma H-4L)C_\eta(t).
 \qedhere
\end{align*}
\end{proof}

\subsection{Commutators and the bilinear estimate}
For $n_r$ defined in \eqref{eq:marginals}, set
\begin{equation}\label{eq:Xnorm}
 X(f)^2=\sup_\eta\int_J\!\int
       n_r(t,x\cdot\omega+\eta)|f(t,x)|^2\dd x\dd t.
\end{equation}
Taking the supremum over fixed longitudinal shifts makes $X$ invariant under spatial translations. It therefore controls convolution operators while retaining the density weight: $X(Tf)\le\|K_T\|_1X(f)$.

\begin{lemma}\label{lem:commutator}
Let $Tf=K_T*f$, where $K_T\in L^1(\R^3)$ is independent of time and
\[
 \int|z\cdot\omega|\,|K_T(z)|\dd z\le C_TL^{-1}.
\]
For $f,g\in C(J;\mathcal S(\R^3))$ and every fixed $\eta\in\R$,
\begin{equation}\label{eq:commutator}
 \int_J|\ip f{[T,q_r^\eta]g}|\dd t
 \le C_TL^{-1}X(f)X(g).
\end{equation}
\end{lemma}
\begin{proof}
The convolution formula for $[T,q_r^\eta]g$ contains the weight difference
\[
 q_r^\eta(\rho-z\cdot\omega)-q_r^\eta(\rho)
 =(z\cdot\omega)\int_0^1
       n_r(t,\rho+\eta-\tau z\cdot\omega)\dd\tau.
\]
For fixed $z,\tau$, Cauchy--Schwarz and the definition of $X$ give
\[
 \int_J\!\int n_r(t,x\cdot\omega+\eta-\tau z\cdot\omega)
       |f(t,x)g(t,x-z)|\dd x\dd t\le X(f)X(g).
\]
After changing variables to $y=x-z$ in the second factor, the weighted norms use shifts $\eta-\tau z\cdot\omega$ and $\eta+(1-\tau)z\cdot\omega$, respectively. Both are covered by the defining supremum. Integrating against $|z\cdot\omega|\,|K_T(z)|\dd z\dd\tau$ proves the estimate.
\end{proof}

\begin{lemma}\label{lem:internalwork}
We have
\begin{equation}\label{eq:internalbound}
 \sup_\eta\int_J|\mathcal I_{q_r^\eta}(w)|\dd t
 \lesssim_{\chi} E_w\CC.
\end{equation}
\end{lemma}
\begin{proof}
Fourier inversion and Cauchy--Schwarz give
\[
 \|v\|_\infty^2
 \lesssim\left(\int_{|\xi|\le2L}|\xi|^{-2}\dd\xi\right)
          \|\nabla w\|_2^2
 \lesssim LE_w.
\]
Since $v=Pw$ and $X(Pf)\lesssim_{\chi}X(f)$,
\begin{align*}
 a^2X(w)^2&\le\CC,&
 \sum_jX(\partial_jw)^2&\le3\CC,& X(v^2)^2&\le2\CC,\\
 X(F(v))&\lesssim LE_wX(w),&
 X(\partial_jF(v))&\lesssim LE_wX(\partial_jw),&
 X(F(v))^2&\lesssim LE_wX(v^2)^2.
\end{align*}
Since $P=\chi(D/L)$, its kernel and that of $P^2$ satisfy
\[
 K_P(z)=L^3\check\chi(Lz),\qquad
 K_{P^2}(z)=L^3\mathcal F^{-1}(\chi^2)(Lz).
\]
Consequently, uniformly for unit vectors $\omega$,
\begin{equation}\label{eq:cutoffmoment}
\begin{split}
 \int |z\cdot\omega|\bigl(|K_P(z)|+|K_{P^2}(z)|\bigr)\dd z
 &\le L^{-1}\int |y|\bigl(|\check\chi(y)|
                 +|\mathcal F^{-1}(\chi^2)(y)|\bigr)\dd y\\
 &\lesssim_\chi L^{-1}.
\end{split}
\end{equation}
The first kernel moment offsets the frequency factor in the low-frequency supremum bound. Retaining the density weight leaves an error proportional to the interaction quantity, with no additional factor of the interval length. Thus~\eqref{eq:commutator}, applied to~\eqref{eq:internalformula}, gives
\begin{align*}
 \sup_\eta\int_J|\mathcal I_{q_r^\eta}|
 &\lesssim L^{-1}\left[
 a^2X(w)X(F(v))+\sum_jX(\partial_jw)X(\partial_jF(v))+X(F(v))^2\right]\\
 &\lesssim E_w\left[a^2X(w)^2+\sum_jX(\partial_jw)^2+X(v^2)^2\right]
 \lesssim E_w\CC.
 \qedhere
\end{align*}
\end{proof}

\begin{proof}[Proof of Proposition~\ref{prop:flux}]
\emph{Step 1: Bound the interaction integral.} Integrate~\eqref{eq:Qprime} and apply Lemma~\ref{lem:internalwork}, using $0\le\mathcal Q_\eta\le M_hE_w$ and $\int_J|\int q_e^\eta s_h|\dd t\le2E_w\|hG\|_{L^1_{t,x}}$. Take the supremum over fixed $\eta$ and use~\eqref{eq:internalbound}:
\[
 [4\pi(\gamma H-4L)-C_\chi E_w]\CC
 \le M_hE_w+2E_w\|hG\|_{L^1_{t,x}}+
       \sup_\eta\int_J|\DD_{q_r^\eta}|.
\]
The separation condition~\eqref{eq:fluxseparation} makes the coefficient on the left at least $\gamma H$. This proves~\eqref{eq:forcedflux}.

\emph{Step 2: Recover the product.} The smoothing estimate~\eqref{eq:productpositive} and pointwise Cauchy--Schwarz give
\begin{align*}
 \|h(\cdot+x_0)T_Kw\|_{L^2_{t,x}}^2
 &\lesssim(1+K/r)\int_J\!\int\rho_h(t,x+x_0)|T_Kw(t,x)|^2,\\
 |T_Kw(t,x)|^2
 &\le\|K_{T_K}\|_1\int|K_{T_K}(z)|\,|w(t,x-z)|^2\dd z.
\end{align*}
For fixed $z$, set $y=x-z$ and apply~\eqref{eq:transversetrace}. Its transverse supremum removes the transverse location of the shifted high-frequency density; integrating that density leaves its longitudinal marginal. The longitudinal displacement enters through $\eta=(x_0+z)\cdot\omega$, giving
\[
 \int_J\!\int\rho_h(t,y+x_0+z)|w(t,y)|^2\dd y\dd t
 \lesssim\log(2+L/a)\int_J C_\eta(t)\dd t
 \le\log(2+L/a)\CC.
\]
Integration against $|K_{T_K}(z)|\dd z$ proves \eqref{eq:fluxproduct}.
\end{proof}

Smooth Heaviside approximations and spatial cutoffs justify the half-space identities. For Schwartz fields on compact time intervals, decay permits removal of the cutoffs by dominated convergence; the differentiated Heaviside approximations converge to the hyperplane measure in distributions. We then apply the Fourier-support argument to the original functions and the weight $q_r^\eta$ satisfying~\eqref{eq:weightbandwidth}. The translations and $\eta$ remain fixed during these limits. Since the estimates are uniform in these parameters, we take their suprema afterward.

The separation condition~\eqref{eq:fluxseparation} absorbs the commutator bound into the positive principal term in~\eqref{eq:principalpositive}. For the frequency components of $u$, it remains to bound $E_w\|hG\|_{L^1_{t,x}}$ and the weighted contribution of $R$.
\section{Nonlinear estimates and frequency induction}\label{sec:recursion}
We estimate the high-frequency forcing and the weighted low-frequency error left by the directional identity. Together with interaction Morawetz for comparable frequencies, these bounds yield the recursion for $\BB$. Throughout this section, assume~\eqref{eq:boot} and~\eqref{eq:roughA}.

\subsection{The frequency-localized nonlinearity}\label{sec:source}
The high-frequency forcing is $Q_HF(u)$. Proposition~\ref{prop:actualflux} requires its product with $Q_Hu$; the interaction Morawetz estimate also requires its product with $H^{-1}\nabla Q_Hu$. Let $Q_H=\varphi(D/H)$, where $\varphi\in C_c^\infty(\R^3)$ is fixed and supported in $\{\xi:c_0\le|\xi|\le C_0\}$, with $0<c_0<C_0<\infty$. This includes $P_H$ and the angular projections used below. By~\eqref{eq:smoothkernel},
\begin{equation}\label{eq:QHkernel}
 \|K_{Q_H}\|_{L^1}
 +H^{-1}\sum_{j=1}^3\|K_{\partial_jQ_H}\|_{L^1}
 \lesssim_\varphi1,\qquad
 Q_Hf=\sum_{M\sim H}Q_HP_Mf,
\end{equation}
where the sum has a uniformly bounded number of terms.
Set $h=Q_Hu$ and $G_H=Q_HF(u)$, so that
$(\ii\partial_t+\Delta)h=G_H$, and define
\begin{equation}\label{eq:DHdef}
 D_H=\sup_{x_0\in\R^3}\int_J\!\int_{\R^3}
 \bigl(|h(t,x+x_0)|+H^{-1}|\nabla h(t,x+x_0)|\bigr)
 |G_H(t,x)|\dd x\dd t.
\end{equation}
\begin{proposition}\label{prop:source}
Assume \eqref{eq:pair}. There are $c>0$ and $H_0\ge c^{-1}$,
depending only on the fixed cutoff functions, such that, for $H\ge H_0$,
\begin{equation}\label{eq:sourcebound}
 D_H\lesssim_s\BB(cH)\frac{\ell_H^3}{H^3m_H^4}
 =\BB(cH)b_H^2\delta_H\ell_H^3.
\end{equation}
\end{proposition}
\begin{proof}
Decompose
\[
 F(u)=\sum_{H_1,H_2,H_3}u_{H_1}\overline{u_{H_2}}u_{H_3}.
\]
Order the three frequencies as $T\ge M\ge K$. There are at most six orderings, and the same estimates apply to each. The support of $\varphi$ forces every nonzero term $Q_H(u_{H_1}\overline{u_{H_2}}u_{H_3})$ to satisfy
\[
 T\gtrsim H,\qquad T\gg H\quad\Longrightarrow\quad M\sim T.
\]
We first estimate the part of $D_H$ containing $|h|$.

If $T\sim H$, writing $Q_H$ as a convolution and applying
Cauchy--Schwarz gives, for each convolution variable $y$,
\begin{align*}
 &\int_J\!\int |h(t,x+x_0)|\,|u_Tu_Mu_K|(t,x-y)\dd x\dd t\\
 &\qquad\le
 \|h(t,x+x_0+y)u_M(t,x)\|_{L^2_{t,x}}
 \|u_Tu_K\|_{L^2_{t,x}}.
\end{align*}
Equations~\eqref{eq:QHkernel},~\eqref{eq:convolutionpair}, and~\eqref{eq:thetastability} permit the use of~\eqref{eq:pair} for both products. Integrating against $|K_{Q_H}(y)|\dd y$ and summing the lower frequencies by~\eqref{eq:sumlow} gives
\[
 \sum_{M,K\lesssim H}
 \frac{C_s\BB(cH)\ell_H}{H^3m_H^2}
       \frac{\theta_{H,M}\theta_{H,K}}{m_Mm_K}
 \lesssim_s\BB(cH)\frac{\ell_H^3}{H^3m_H^4}.
\]
If $T\gg H$, denote the three frequencies by $T\ge T'\ge M$;
then $T'\sim T$. Apply \eqref{eq:pair} to $u_Th$ and $u_{T'}u_M$.
By \eqref{eq:sumlow} and \eqref{eq:sumthree},
\begin{align*}
 \sum_{T\gg H}\sum_{M\le T}
 \frac{C_s\BB(cH)\ell_T}{T^3m_T^2m_Hm_M}
       \theta_{T,H}\theta_{T,M}
 &\lesssim_s\frac{\BB(cH)}{m_H}
       \sum_{T\gg H}\frac{\ell_T^2}{T^3m_T^3}\\
 &\lesssim_s\BB(cH)\frac{\ell_H^2}{H^3m_H^4}.
\end{align*}
The same proof applies with $h$ replaced by $H^{-1}\partial_jh$,
using \eqref{eq:QHkernel}. Summing over $j$ proves the result.
\end{proof}

\subsection{The low-frequency energy and its error}\label{sec:exterior}
The directional estimate~\eqref{eq:forcedflux} needs both a bound for the low-frequency energy and a bound for its weighted forcing error. Set $a=\mu^{-1}$, $P=P_{\le L}$, and
\begin{equation}\label{eq:wvr}
 w=Pu,\qquad v=P^2u=Pw,\qquad R_L=P\bigl(F(u)-F(v)\bigr).
\end{equation}
Then
\begin{equation}\label{eq:lowforced}
 (\ii\partial_t+\Delta)w=PF(v)+R_L.
\end{equation}
The cancellation in Section~\ref{sec:flux} uses $v=Pw$, so for $w=Pu$ we retain the second projection in $v$; $P^2\ne P$. For $L\gg1$, the auxiliary energy~\eqref{eq:abstractenergy} is
\[
 E_L=\sup_{t\in J}\left(a^2\|w(t)\|_2^2+\|\nabla w(t)\|_2^2
                         +\frac{1}{2}\|v(t)\|_4^4\right).
\]
We first bound $E_L$ without a positive power of $\mu$.

\begin{lemma}\label{lem:actualEL}
Under \eqref{eq:boot}, for all $L\gg1$, including $L>N$,
\begin{equation}\label{eq:actualEL}
 E_L\lesssim_s m_L^{-2},
 \qquad
 \norm{P_{\le L}^2u}_{L^\infty_{t,x}}^2\lesssim_s Lm_L^{-2}.
\end{equation}
\end{lemma}
\begin{proof}
The mass and kinetic terms satisfy
\[
 a^2\|P_{\le L}u\|_2^2\lesssim1,\qquad
 \|\nabla P_{\le L}u\|_2^2\lesssim_s m_L^{-2}\|\nabla Iu\|_2^2
 \lesssim_s m_L^{-2}.
\]
The symbol of $P_{\le N}I_N^{-1}$ is a fixed smooth function of
$\xi/N$. Thus \eqref{eq:smoothkernel} gives
\[
 \|K_{P_{\le L}^2P_{\le N}I_N^{-1}}\|_1
 \le\|K_{P_{\le L}}\|_1^2\|K_{P_{\le N}I_N^{-1}}\|_1
 \lesssim_s1.
\]
The critical estimate~\eqref{eq:criticalrough} and the kinetic bound give the high-frequency $L^3_x$ and $L^6_x$ estimates, respectively. Interpolation controls the high-frequency contribution to the potential energy:
\begin{align*}
 \|v_{\le N}\|_4&\lesssim_s\|Iu\|_4\lesssim_s1,\\
 \|v_{>N}\|_3&\lesssim_sN^{-\frac{1}{2}},&
 \|v_{>N}\|_6&\lesssim_sm_L^{-1},\\
 \|v_{>N}\|_4^4&\le\|v_{>N}\|_3^2\|v_{>N}\|_6^2
                  \lesssim_sN^{-1}m_L^{-2}.
\end{align*}
Combining these bounds and using \eqref{eq:sumLinfty}, we obtain
\[
 E_L\lesssim_s1+m_L^{-2}+N^{-1}m_L^{-2}\lesssim_sm_L^{-2},\qquad
 \|v\|_\infty\lesssim_s
       \sum_{K\le4L}\frac{\sqrt K}{m_K}
       \lesssim_s\frac{\sqrt L}{m_L}.
\qedhere
\]
\end{proof}
For the functions in \eqref{eq:wvr}, write
$\DD_{q,L}=\DD_q(w,R_L)$, where $\DD_q$ is defined in
\eqref{eq:abstractwork}. Thus
\begin{equation}\label{eq:exteriorD}
 \DD_{q,L}(t)=2\Imn\int_{\R^3}q(t,x)
 \left[a^2\bar wR_L+\nabla\bar w\cdot\nabla R_L
       +\overline{F(v)}PR_L\right]\dd x.
\end{equation}
The weighted error satisfies the following estimate.

\begin{proposition}\label{prop:exterior}
Let $0<c\le\frac{1}{128}$ and $L$ satisfy $cL\ge1$, with $a=\mu^{-1}$.
Assume \eqref{eq:pair} and the pointwise bound in
\eqref{eq:massfreq}. If $q\in L^\infty(J\times\R^3;\R)$ satisfies
\begin{equation}\label{eq:narrowweight}
 \supp\widehat q(t)\subset B(0,L/256)
 \quad\text{for almost every }t,
\end{equation}
then
\begin{equation}\label{eq:exteriorbound}
 \int_J|\DD_{q,L}(t)|\dd t
 \lesssim_s\|q\|_{L^\infty_{t,x}}\BB(cL)
 \frac{\ell_L^3}{Lm_L^4}(1+\delta_L).
\end{equation}
\end{proposition}

Expand the quartic and sextic contributions to $\DD_{q,L}$ in Proposition~\ref{prop:exterior}. Each contains $b=u-v$ or its conjugate, so its largest dyadic frequency $T$ is at least $L/2$. The narrow Fourier support of $q$ forces a second frequency of size at least $T/64$ whenever the weighted integral is nonzero. We apply this support restriction before taking absolute values, then estimate the four main factors by two bilinear bounds. The resulting sum is recorded below.

\begin{lemma}\label{lem:fourbook}
Assume \eqref{eq:pair}. For fixed $0<c\le\frac{1}{128}$ and $L$ with
$cL\ge1$,
\begin{equation}\label{eq:AL}
 \begin{split}
 &\sum_{\substack{T\ge T'\ge M\ge K\\T\ge L/2,\ T'\ge T/64}}
 \left(\sup_y\|u_T(t,x+y)u_M(t,x)\|_{L^2_{t,x}}\right)
 \left(\sup_z\|u_{T'}(t,x+z)u_K(t,x)\|_{L^2_{t,x}}\right)\\
 &\hspace{25mm}\lesssim_s
 \mathcal A_L:=\BB(cL)\frac{\ell_L^3}{L^3m_L^4}.
 \end{split}
\end{equation}
\end{lemma}
\begin{proof}
Since $T'\sim T$ and $T'\ge cL$, \eqref{eq:pair} and
\eqref{eq:thetastability} bound each summand by
\[
 \frac{C_s\BB(cL)\ell_T}{T^3m_T^2}
       \frac{\theta_{T,M}\theta_{T,K}}{m_Mm_K}.
\]
For each $T$, there are only finitely many possible dyadic $T'$, with
a uniform bound on their number. By \eqref{eq:sumlow} and
\eqref{eq:sumfour}, the sum is bounded by
\begin{align*}
 C_s\BB(cL)\sum_{T\ge L/2}
 \frac{\ell_T}{T^3m_T^2}
       \left(\sum_{M\le T}\frac{\theta_{T,M}}{m_M}\right)^2
 &\lesssim_s\BB(cL)\sum_{T\ge L/2}\frac{\ell_T^3}{T^3m_T^4}\\
 &\lesssim_s\BB(cL)\frac{\ell_L^3}{L^3m_L^4}.
 \qedhere
\end{align*}
\end{proof}

\begin{proof}[Proof of Proposition~\ref{prop:exterior}]
\emph{Step 1: Frequency support.} Set $b=u-v$. Then
\begin{equation}\label{eq:expand}
 \begin{gathered}
 \supp\widehat b\subset\{|\xi|\ge L\},\\
 F(u)-F(v)=2|v|^2b+v^2\bar b+2v|b|^2+\bar vb^2+|b|^2b.
 \end{gathered}
\end{equation}
Each term contains $b$ or $\bar b$. Decompose all factors into dyadic
frequencies $D_1\ge\cdots\ge D_m$, with $m=4$ or $6$, and put
$T=D_1$. A nonzero factor from $b$ has frequency at least $L/2$,
so $T\ge L/2$. If $D_2<T/64$, then for any conjugation signs
$\epsilon_j\in\{1,-1\}$,
\[
 \left|\sum_{j=1}^m\epsilon_j\xi_j\right|
 \ge\frac T2-\frac{2(m-1)T}{64}\ge\frac{11T}{32}.
\]
Let $\mathcal T_{\vec D}$ denote the corresponding expression before
multiplication by $q$. Translations and Fourier multipliers preserve
this support restriction. By \eqref{eq:narrowweight},
\[
 P_{\le T/8}q=q,\qquad P_{\le T/8}\mathcal T_{\vec D}=0,
 \qquad
 \int q\mathcal T_{\vec D}
 =\int qP_{\le T/8}\mathcal T_{\vec D}=0.
\]
These equalities use distributional duality and remain valid for bounded
$q$, including weights depending on one spatial variable.
Thus every nonzero term has
\begin{equation}\label{eq:twolargefrequencies}
 T\ge L/2,\qquad D_2\ge T/64.
\end{equation}
We use this restriction before taking absolute values.

\emph{Step 2: Quartic terms.} Write $\DD_{q,L}=\DD_{q,L}^{(4)}+\DD_{q,L}^{(6)}$. The first two terms of~\eqref{eq:exteriorD} are quartic and the last is sextic. The normalized derivative multipliers have uniformly bounded kernels by~\eqref{eq:smoothkernel}. Since
\[
 \nabla w=L(L^{-1}\nabla P)u,\qquad
 \nabla R_L=L(L^{-1}\nabla P)(F(u)-F(v)),\qquad a^2\le L^2,
\]
the two extracted factors of $L$ and the comparison $a^2\le L^2$ give the common prefactor $L^2$. The four-factor version of~\eqref{eq:producttreekernel} and Lemmas~\ref{lem:kernel} and~\ref{lem:fourbook} then control the translated integrals:
\begin{equation}\label{eq:exteriorquartic}
 \int_J|\DD_{q,L}^{(4)}(t)|\dd t
 \lesssim_s\|q\|_{L^\infty_{t,x}} L^2\mathcal A_L.
\end{equation}
\emph{Step 3: Sextic terms.} Place the two lowest frequencies in $L^\infty_{t,x}$ using the bounds
\[
 v_D=P^2u_D,\qquad b_D=(1-P^2)u_D,\qquad
 \|v_D\|_{L^\infty_{t,x}}+\|b_D\|_{L^\infty_{t,x}}
 \lesssim_s\frac{\sqrt D}{m_D},
\]
by \eqref{eq:massfreq} and the uniform total variation bounds for
the kernels of $P^2$ and $1-P^2$. Order the six frequencies as
\[
 T\ge T'\ge M\ge K\ge S_1\ge S_2.
\]
Three factors belong to $\overline{F(v)}$ and have dyadic frequencies
at most $4L$. Hence $S_1,S_2\le4L$, while
\eqref{eq:twolargefrequencies} gives $T'\sim T\gtrsim L$.
Use \eqref{eq:producttreekernel} to write each term as an integral
of six translated factors against a measure $\nu_{\vec D}$ with
$\|\nu_{\vec D}\|_{\TV}\lesssim1$. For fixed translations, put
the factors at $S_1,S_2$ in $L^\infty_{t,x}$ and apply
Cauchy--Schwarz to the products at $(T,M)$ and $(T',K)$.
Integrating against $|\nu_{\vec D}|$ and summing, we obtain
\begin{align*}
 \int_J|\DD_{q,L}^{(6)}(t)|\dd t
 &\lesssim_s \|q\|_{L^\infty_{t,x}}
 \sum_{S_2\le S_1\le4L}\frac{\sqrt{S_1S_2}}{m_{S_1}m_{S_2}}
 \sum_{\substack{T\ge T'\ge M\ge K\ge S_1\\
                   T\ge L/2,\ T'\ge T/64}}\\[-1mm]
 &\qquad\times
 \left(\sup_y\|u_T(t,x+y)u_M(t,x)\|_{L^2_{t,x}}\right)
 \left(\sup_z\|u_{T'}(t,x+z)u_K(t,x)\|_{L^2_{t,x}}\right).
\end{align*}
The finite number of monomials and frequency orderings affects only the constant. For fixed $S_1,S_2$, enlarge the remaining nonnegative sum and apply~\eqref{eq:AL}. Equation~\eqref{eq:sumLinfty} then controls the two lowest frequencies:
\begin{equation}\label{eq:exteriorsextic}
 \begin{split}
 \int_J|\DD_{q,L}^{(6)}(t)|\dd t
 &\lesssim_s\|q\|_{L^\infty_{t,x}}\mathcal A_L
       \left(\sum_{S\le4L}\frac{\sqrt S}{m_S}\right)^2\\
 &\lesssim_s\|q\|_{L^\infty_{t,x}} Lm_L^{-2}\mathcal A_L.
 \end{split}
\end{equation}
Combining the two estimates proves
\[
 \int_J|\DD_{q,L}(t)|\dd t
 \lesssim_s\|q\|_{L^\infty_{t,x}}(L^2+Lm_L^{-2})\mathcal A_L
 =\|q\|_{L^\infty_{t,x}}\BB(cL)\frac{\ell_L^3}{Lm_L^4}(1+\delta_L).
 \qedhere
\]
\end{proof}

\subsection{Comparable frequencies}
For comparable frequencies, H\"older reduces the product estimate to an $L^4_{t,x}$ bound for each factor. We use the tensor-product interaction Morawetz identity~\cite{CGT} for $(\ii\partial_t+\Delta)h=G$ and retain all forcing terms. In the application $h=Q_Hu$, the source is $G=Q_HF(u)$ and need not equal $F(h)$.

\begin{lemma}\label{lem:forcedIM}
Let $(\ii\partial_t+\Delta)h=G$ on a finite interval $J$, with $h,G$
smooth and rapidly decreasing in space, and
$\supp\widehat h(t)\subset\{|\xi|\le C_hH\}$. Set
$M_h=\sup_J\norm h_2^2$. Then
\begin{equation}\label{eq:forcedIM}
 \norm h_{L^4_{t,x}(J)}^4
 \lesssim H M_h^2
 +H M_h\norm{\bar hG}_{L^1_{t,x}(J)}
 +M_h\norm{\nabla\bar h\,G}_{L^1_{t,x}(J)}.
\end{equation}
The implicit constant depends only on $C_h$.
\end{lemma}
\begin{proof}
Set $\rho=|h|^2$, $j=\Imn(\bar h\nabla h)$, and
\[
 \mathcal M=2\iint\rho(y)\frac{x-y}{|x-y|}\cdot j(x)\dd x\dd y,
 \qquad |\mathcal M|\le2\|h\|_2^3\|\nabla h\|_2\lesssim HM_h^2 .
\]
Set $Z(x,y)=h(x)h(y)$, $a(x,y)=|x-y|$, and
$\mathcal G=G(x)h(y)+h(x)G(y)$.
Applying the virial identity in \cite[Lemma~5.3]{Visan07} to $Z$ on
$\R^6$ gives
\begin{equation}\label{eq:tensorvirial}
 \frac{\dd}{\dd t}\left(2\int\nabla a\cdot\Imn(\bar Z\nabla Z)\right)
 =-\int\Delta_6^2a\,|Z|^2
 +4\int a_{jk}\Ren(\partial_j\bar Z\,\partial_kZ)
 +2\int\nabla a\cdot\{\mathcal G,Z\}_p .
\end{equation}
Here derivatives are taken in $\R^6$, $a_{jk}=\partial_j\partial_k a$,
repeated indices are summed, and
$\{f,g\}_p=2\Ren(f\nabla\bar g)-\nabla\Ren(f\bar g)$.
Moreover,
\[
 2\int\nabla a\cdot\Imn(\bar Z\nabla Z)=2\mathcal M,\qquad
 -\Delta_6^2a=32\pi\delta(x-y),\qquad (a_{jk})\ge0 .
\]
After integration in time, the diagonal term supplies $32\pi\|h\|_{L^4_{t,x}(J)}^4$.
To justify the identity for $a(x,y)=|x-y|$, we use the smooth approximation
\[
 a_\varepsilon(x,y)=\sqrt{|x-y|^2+\varepsilon^2},\qquad
 -\Delta_6^2a_\varepsilon
 =\frac{60\varepsilon^4}{(|x-y|^2+\varepsilon^2)^{\frac{7}{2}}}
 \to32\pi\delta(x-y).
\]
Its Hessian is nonnegative, $|\nabla_xa_\varepsilon|\le1$, and
$\Delta_xa_\varepsilon\lesssim|x-y|^{-1}$, uniformly in $\varepsilon$.
The $x$-component of the bracket is
\[
 \{\mathcal G,Z\}_{p,x}
 =\rho(y)\{G,h\}_p(x)+2\Imn(\bar hG)(y)j(x).
\]
Integrating the derivative in the bracket by parts and using symmetry,
we obtain
\begin{align*}
 \left|\int\nabla a\cdot\{\mathcal G,Z\}_p\right|
 &\lesssim M_h\|\nabla\bar h\,G\|_{L^1_x}\\
 &\quad+\left(\sup_x\int\frac{|h(y)|^2}{|x-y|}\dd y
                  +\|h\|_2\|\nabla h\|_2\right)\|\bar hG\|_{L^1_x}.
\end{align*}
Hardy's inequality and Bernstein's inequality then give
\begin{align*}
 \sup_x\int\frac{|h(y)|^2}{|x-y|}\dd y
 &\lesssim\|h\|_2\|\nabla h\|_2\lesssim HM_h,\\
 \left|\int\nabla a\cdot\{\mathcal G,Z\}_p\right|
 &\lesssim HM_h\|\bar hG\|_{L^1_x}
                    +M_h\|\nabla\bar h\,G\|_{L^1_x}.
\end{align*}
Integrate \eqref{eq:tensorvirial}, discard the nonnegative Hessian term,
and let $\varepsilon\to0^+$ to obtain \eqref{eq:forcedIM}.
\end{proof}

Taking $h=Q_Hu$ and $G=Q_HF(u)$, Proposition~\ref{prop:source}
bounds both terms containing $G$ in \eqref{eq:forcedIM}.
\begin{corollary}\label{cor:diagonal}
Under \eqref{eq:boot}, let $Q_H$ be as in Section~\ref{sec:source}
and $H\ge H_0$. Then
\begin{equation}\label{eq:diagonal}
 \norm{Q_Hu}_{L^4_{t,x}(J)}^4
 \lesssim_s\frac1{H^3m_H^4}
 \left[1+\BB(cH)\delta_H\ell_H^3\right].
\end{equation}
\end{corollary}
\begin{proof}
Since $M_h\lesssim b_H^2$, \eqref{eq:forcedIM} and
\eqref{eq:sourcebound} give
\begin{align*}
 \|Q_Hu\|_{L^4_{t,x}}^4
 &\lesssim_s Hb_H^4+Hb_H^2D_H\\
 &\lesssim_s\frac1{H^3m_H^4}
       [1+\BB(cH)\delta_H\ell_H^3].
\qedhere
\end{align*}
\end{proof}

H\"older's inequality now controls products at comparable frequencies,
uniformly under fixed spatial translations of either factor.

\subsection{Separated frequencies}
We choose the cutoff frequency $L$ to separate the high and low frequencies while keeping the thresholds in the nonlinear estimates bounded below by a fixed multiple of $R$. Thus, for $1\le R\le\min\{H,N\}$, we require $K\lesssim L\ll H$ and $L\gtrsim R$, even when $K$ is arbitrarily small.

Choose a finite smooth partition of unity $\sum_{\nu=1}^{J_*}\chi_\nu=1$
on the unit sphere, with directions $\omega_\nu$ and a fixed $\gamma>0$
such that
\[
 \supp\chi_\nu\subset\{\theta\in\mathbb S^2:
                 \theta\cdot\omega_\nu\ge2\gamma\},\qquad
 Q_H^{(\nu)}=\chi_\nu(D/|D|)P_H.
\]
The six signed coordinate directions suffice for sufficiently small
fixed $\gamma$. These multipliers have the form in
Section~\ref{sec:source}, and
\[
 \xi\cdot\omega_\nu\ge\gamma H
 \quad\text{on }\supp\widehat{Q_H^{(\nu)}u},\qquad
 \|u_H(\cdot+x_0)u_K\|_{L^2_{t,x}(J)}^2
 \le J_*\sum_{\nu=1}^{J_*}
       \|(Q_H^{(\nu)}u)(\cdot+x_0)u_K\|_{L^2_{t,x}(J)}^2.
\]
Formula~\eqref{eq:QHkernel} gives uniform kernel bounds for this fixed
finite family. Fix one component and write $h=Q_Hu$ and
$\omega=\omega_\nu$.

With the multipliers and smoothing kernel fixed, choose $C_*$ and then $\Gamma$ to make the $L/H$ contribution to~\eqref{eq:fluxseparation} sufficiently small. Whenever $H/K\ge\Gamma$, let $L$ be the smallest dyadic frequency satisfying
\begin{equation}\label{eq:Lchoice}
 L\ge\max\{4K,\min(H,N)/C_*\},
\end{equation}
and set $r=\eps L$.
Then $L/H\le\max\{8/\Gamma,2/C_*\}$.
We next choose $\eps$ so that $2r\le L/256$.
The constants $C_*,\Gamma,\eps$ are fixed before $N$ is chosen.
The resulting $L,r$ satisfy
\begin{equation}\label{eq:scaleproperties}
 2K\le L,\qquad 6L\le\gamma H/2,\qquad
 2r\le L/256,\qquad K/r\le C_\eps,
 \qquad m_L\sim_s m_K.
\end{equation}
Either $L\sim K$, or both $K$ and $L$ are below $N$, where their weights equal one. In addition,
\begin{equation}\label{eq:thresholdproperty}
 H\ge R,\quad R\le N
 \quad\Longrightarrow\quad
 L\ge cR
\end{equation}
for a fixed $c>0$.

With this $L$, take $w=P_{\le L}u$ as in \eqref{eq:wvr}, and let
$\CC_{H,L}$ be the quantity in \eqref{eq:fluxC} for these choices
of $h,w,r,\omega$. With $G=Q_HF(u)$, the interaction weights satisfy
\begin{equation}\label{eq:twoweights}
 \norm{q_e^\eta}_{L^\infty_{t,x}}\le E_L,
 \qquad
 \norm{q_r^\eta}_{L^\infty_{t,x}}\lesssim_s b_H^2,
 \qquad
 \supp\widehat{q_r^\eta}\subset B(0,L/256).
\end{equation}
Proposition~\ref{prop:source} bounds $\|hG\|_{L^1_{t,x}}$, while
Proposition~\ref{prop:exterior} bounds the weighted term
$\DD_{q_r^\eta}(w,R_L)$ in \eqref{eq:forcedflux}.

\begin{proposition}\label{prop:actualflux}
Fix $\kappa\in(0,1)$. There is $N_{\mathrm{flux}}=N_{\mathrm{flux}}(s,\kappa)$,
independent of $A,\mu,J$, such that, if $N\ge N_{\mathrm{flux}}$,
$N^\kappa\le R\le\min\{H,N\}$ and $H/K\ge\Gamma$, then
\begin{equation}\label{eq:fluxfeedback}
 \begin{split}
 H\CC_{H,L}\lesssim_s b_H^2m_L^{-2}\bigl[1
 &+\BB(cH)\delta_H\ell_H^3\\
 &+\BB(cL)\delta_L\ell_L^3(1+\delta_L)\bigr].
 \end{split}
\end{equation}
Furthermore,
\begin{equation}\label{eq:actualreconstruction}
 \norm{h(t,x+x_0)u_K(t,x)}_{L^2_{t,x}(J)}^2
 \lesssim\log(2+L\mu)\CC_{H,L}
\end{equation}
for every fixed $x_0$.
\end{proposition}
\begin{proof}
Decreasing $c$ if necessary, we use the same fixed constant in
Propositions~\ref{prop:source} and \ref{prop:exterior}; their
required lower bounds on $H$ and $L$ hold for sufficiently large $N$ by
\eqref{eq:thresholdproperty}. By \eqref{eq:sourcebound},
\eqref{eq:exteriorbound}, and \eqref{eq:twoweights},
\begin{align*}
 M_hE_L&\lesssim_sb_H^2m_L^{-2},\\
 E_L\|hG\|_{L^1_{t,x}}
 &\lesssim_sm_L^{-2}\BB(cH)b_H^2\delta_H\ell_H^3,\\
 \sup_\eta\int_J|\DD_{q_r^\eta}(w,R_L)|
 &\lesssim_sb_H^2\BB(cL)\frac{\ell_L^3}{Lm_L^4}(1+\delta_L).
\end{align*}

The fixed choices of $C_*$ and $\Gamma$ make $L/H$ small enough. To verify the remaining energy-separation condition, use Lemma~\ref{lem:actualEL},~\eqref{eq:thresholdproperty}, and~\eqref{eq:sumdelta} with $p=0$:
\[
 E_L/H\lesssim_sm_L^{-2}/H\le\delta_L
 \lesssim_s R^{-1}\le N^{-\kappa}.
\]
Choose $N_{\mathrm{flux}}(s,\kappa)$ large enough to ensure~\eqref{eq:fluxseparation} and all fixed lower-frequency thresholds, independently of $A$ and $\mu$. The estimate~\eqref{eq:forcedflux} then gives~\eqref{eq:fluxfeedback}. Since $2K\le L$, the symbol of $P_{\le L}$ is one on the support of $P_K$, so $u_K=P_Kw$. Combining $K/r\le C_\eps$ and $a=\mu^{-1}$ with~\eqref{eq:fluxproduct} proves~\eqref{eq:actualreconstruction}.
\end{proof}

\subsection{Frequency induction}\label{sec:iteration}
The nonlinear estimates involve $\BB(cH)$ and $\BB(cL)$. Our choice of $L$ makes both thresholds at least a fixed multiple of $R$, while Lemma~\ref{lem:sums} bounds their coefficients by $C_s\ell_N^3/R$. Corollary~\ref{cor:diagonal} handles comparable frequencies and Proposition~\ref{prop:actualflux} handles separated frequencies, giving a recursion for $\BB(R)$. For $R\ge N^\kappa$, its coefficient becomes small for sufficiently large $N$ under the polynomial parameter bound imposed below. For the related radial induction, see~\cite[Section~4]{Dodson19}.
\begin{proposition}[Frequency recursion]\label{prop:recursion}
Fix $\frac{1}{2}<s<1$ and $\kappa\in(0,1)$. Under \eqref{eq:boot} and
\eqref{eq:roughA}, there are $c_1\in(0,1)$ and $C_s<\infty$,
depending only on $s$ and the fixed multipliers, such that
\begin{equation}\label{eq:recursion}
 \BB(R)\le C_s+C_s\frac{\ell_N^3}{R}\BB(c_1R)
\end{equation}
whenever $N\ge N_{\mathrm{rec}}$ and $N^\kappa\le R\le N$,
where $N_{\mathrm{rec}}=N_{\mathrm{rec}}(s,\kappa)$ depends only
on $s,\kappa$ and the fixed multiplier choices.
\end{proposition}
\begin{proof}
Choose $c_1$ sufficiently small for all the fixed threshold
reductions below, independently of $N,A,\mu,J$ and $\kappa$.
Choose $N_{\mathrm{rec}}$ large enough for Lemma~\ref{lem:local},
Proposition~\ref{prop:actualflux}, and $c_1N_{\mathrm{rec}}^\kappa\ge1$.
Enlarge it so that $N_{\mathrm{rec}}^\kappa/\Gamma$ and
$N_{\mathrm{rec}}^\kappa/C_*$ exceed all fixed lower-frequency
thresholds in Corollary~\ref{cor:diagonal} and
Propositions~\ref{prop:source} and \ref{prop:exterior}.
Fix $H\ge R$, $K\le H$, and $x_0\in\R^3$.

\emph{Step 1: Comparable frequencies.} If $H/K\le\Gamma$, set
\[
 X=\BB(cH/\Gamma),\qquad
 d=\max(\delta_H\ell_H^3,\delta_K\ell_K^3).
\]
Our choice of $c_1$ ensures $cH/\Gamma\ge c_1R$, so
$X\le\BB(c_1R)$. Corollary~\ref{cor:diagonal} gives
\begin{align*}
 \|u_H(\cdot+x_0)u_K\|_{L^2_{t,x}}^2
 &\le\|u_H\|_{L^4_{t,x}}^2\|u_K\|_{L^4_{t,x}}^2\\
 &\lesssim_s
 \frac{\sqrt{1+X\delta_H\ell_H^3}\sqrt{1+X\delta_K\ell_K^3}}
       {H^{\frac{3}{2}}K^{\frac{3}{2}}m_H^2m_K^2}\\
 &\le\frac{C_s(1+Xd)}{H^{\frac{3}{2}}K^{\frac{3}{2}}m_H^2m_K^2},\\
 d&\lesssim_{s,\Gamma}\ell_N^3/R,\qquad
 (H/K)^{\frac{3}{2}}/\ell_H\le\Gamma^{\frac{3}{2}}.
\end{align*}
\emph{Step 2: Separated frequencies.} For $H/K\ge\Gamma$, Proposition~\ref{prop:actualflux} gives
\begin{equation}\label{eq:normalization}
 \frac{H^3m_H^2m_K^2}{\ell_H}
       \frac{b_H^2m_L^{-2}}H\log(2+L\mu)
 =\frac{m_K^2}{m_L^2}\frac{\log(2+L\mu)}{\ell_H}\lesssim_s1.
\end{equation}
Moreover, \eqref{eq:thresholdproperty} and Lemma~\ref{lem:sums} give
\[
 cH,cL\ge c_1R,\qquad
 \max(\delta_H\ell_H^3,\delta_L\ell_L^3(1+\delta_L))
       \lesssim_s\ell_N^3/R.
\]
Because $\BB$ is nonincreasing, each occurrence of the bilinear constant on the right is bounded by $\BB(c_1R)$. Sum over the fixed finite angular partition, then take the suprema over $H,K,x_0$ to obtain~\eqref{eq:recursion}.
\end{proof}

The recursion imposes no positive lower bound on $K$ and no upper
bound on $H$. If $H\gg N$ and $K\ll1$, the choice
\eqref{eq:Lchoice} gives $L\sim N/C_*$; if $K>N$, it gives
$L\sim K$ and $m_L\sim m_K$.

To complete the iteration, assume that $C_0>0$ and $p_0\ge0$ are
independent of $N,J$ and
\begin{equation}\label{eq:polynomialparameters}
 A+\mu\le C_0N^{p_0}.
\end{equation}
The polynomial parameter bound gives $\ell_N\le C(C_0,p_0)\log(2+N)$. During the iteration, $u,J,N,A,\mu$ remain fixed. Only the threshold in $\BB(R)$ decreases. At its terminal value, Lemma~\ref{lem:pair} gives the finite rough bound $\BB(R)\lesssim_s A^4$; the product of the recurrence coefficients suppresses this polynomial loss. Section~\ref{sec:completion} supplies~\eqref{eq:polynomialparameters}.

\begin{proposition}[Uniform high-frequency bilinear bound]\label{prop:highpair}
Fix $c_E\in(0,1)$ and $\kappa\in(0,1)$. Under \eqref{eq:boot},
\eqref{eq:roughA} and \eqref{eq:polynomialparameters}, for all
$N\ge N_0(s,C_0,p_0,c_E,\kappa)$,
\begin{equation}\label{eq:Bbounded}
 \BB(c_EN)\le C_s.
\end{equation}
In particular, uniformly over fixed relative translations and all $H\ge c_EN$, $K\le H$,
\begin{equation}\label{eq:highpair}
 \norm{u_H(t,x+x_0)u_K(t,x)}_{L^2_{t,x}(J)}^2
 \le C_s\frac{\ell_H}{H^3m_H^2m_K^2}\theta_{H,K}^2.
\end{equation}
The threshold $N_0$ is independent of $|J|$ and of higher Sobolev norms.
\end{proposition}
\begin{proof}
Increase $N_0=N_0(s,C_0,p_0,c_E,\kappa)$ so that
$N_0\ge N_{\mathrm{rec}}$, $c_EN\ge N^\kappa$, and
$c_1N^\kappa\ge1$ whenever $N\ge N_0$. Set
\[
 R_j=c_1^jc_EN,\qquad
 k=\min\{j:R_j<N^\kappa\},\qquad
 q_N=C_s\ell_N^3N^{-\kappa},
\]
where $C_s$ is the coefficient in \eqref{eq:recursion}.
The number of steps satisfies $k\sim\log N$, with constants depending only on $c_1,c_E,\kappa$, and the terminal threshold obeys $R_k\ge c_1N^\kappa\ge1$. For $j<k$, the recursion applies at $R_j$ with coefficient at most $q_N$. Since $\ell_N\lesssim_{C_0,p_0}\log(2+N)$, enlarge $N_0$ so that
\[
 q_N\le N^{-\kappa/2}\le\frac{1}{2}.
\]
Iterating \eqref{eq:recursion} and using
Lemma~\ref{lem:pair} at the last threshold gives
\begin{align*}
 \BB(c_EN)
 &\le C_s\sum_{j=0}^{k-1}q_N^j+q_N^k\BB(R_k)\\
 &\le2C_s+C_sA^4N^{-\kappa k/2}\\
 &\le2C_s+C_sC_0^4N^{4p_0-\kappa k/2}.
\end{align*}
Since $k\sim\log N$, the last term tends to zero as $N\to\infty$. Increasing $N_0$ once more gives~\eqref{eq:Bbounded}. These choices depend only on $s,C_0,p_0,c_E,\kappa$ and the fixed cutoffs, independently of the length of $J$ and of higher Sobolev norms. Lemma~\ref{lem:pair} with $R=c_EN$ now yields~\eqref{eq:highpair}.
\end{proof}

The same $A,\mu,N$ apply on every subinterval of $J$. Combining
\eqref{eq:Bbounded} with Proposition~\ref{prop:energy}, with
$c_E\le2^{-16}$ fixed as there, gives
\begin{equation}\label{eq:closedenergy}
 \int_{J'}\left|\frac{\dd}{\dd t}E(I_Nu(t))\right|\dd t
 \lesssim_s\frac{\ell_N^3}{N}+\frac{\ell_N}{N^2},
 \qquad J'\subset J.
\end{equation}
Under \eqref{eq:polynomialparameters}, the right-hand side tends
to zero as $N\to\infty$, uniformly in the length of $J'$.

\section{Global existence and scattering}\label{sec:completion}
We prove Theorem~\ref{thm:main} using the uniform almost conservation estimate~\eqref{eq:closedenergy}, scaling, and continuity. Global spacetime bounds then give scattering. First take $\frac{1}{2}<s<1$ and Schwartz data, whose solutions are global and smooth~\cite{Cazenave,GV}. Uniform bounds on $H^s$ balls allow passage to rough data; persistence of regularity gives the higher-regularity extension.

\subsection{Scaling and the uniform bound}
\begin{proposition}\label{prop:globalHs}
Let $\frac{1}{2}<s<1$ and $M\ge2$. Solutions with Schwartz initial data
satisfying $\|u_0\|_{H^s}\le M$ obey \eqref{eq:uniformball}, with a
constant depending only on $s$ and $M$.
\end{proposition}
\begin{proof}
\emph{Step 1: Normalize the initial energy.} The scaling of~\eqref{eq:NLS} gives
\[
 \|u_{0,\lambda}\|_2=\lambda^{-\frac{1}{2}}\|u_0\|_2,\qquad
 \|u_{0,\lambda}\|_{\dot H^s}
       =\lambda^{s-\frac{1}{2}}\|u_0\|_{\dot H^s},\qquad
 \|u_{0,\lambda}\|_{\dot H^{\frac{1}{2}}}=\|u_0\|_{\dot H^{\frac{1}{2}}}.
\]
By Sobolev embedding and interpolation,
\[
 \|I_Nf\|_4^4\le\|I_Nf\|_3^2\|I_Nf\|_6^2
 \lesssim\|f\|_{\dot H^{\frac{1}{2}}}^2\|\nabla I_Nf\|_2^2.
\]
Consequently,
\begin{equation}\label{eq:initialenergy}
 E(I_Nu_{0,\lambda})
 \lesssim_s N^{2(1-s)}\lambda^{2s-1}M^2(1+M^2).
\end{equation}
Fix $\kappa=\frac{1}{2}$, $c_E\le2^{-16}$, and choose $d_{s,M}>0$
sufficiently small that
\begin{equation}\label{eq:lambdachoice}
 \lambda=d_{s,M}N^{-2(1-s)/(2s-1)},\qquad
 E(I_Nu_{0,\lambda})\le\frac{1}{2}.
\end{equation}
Use the same parameters
\[
 \mu_\lambda=\max\{1,\lambda^{-\frac{1}{2}}M\},\qquad
 A=C_s(1+\mu_\lambda^3)
\]
for every datum in the ball and every interval on which the modified energy is at most one. Choose the constant in $A$ as in~\eqref{eq:roughA}. These common choices give
\begin{equation}\label{eq:scaledparameters}
 \mu_\lambda\lesssim_{s,M}N^{(1-s)/(2s-1)},\qquad
 A\lesssim_{s,M}N^{3(1-s)/(2s-1)},\qquad
 \ell_N\lesssim_{s,M}\log(2+N).
\end{equation}
\emph{Step 2: Fix the smoothing scale and close the energy bound.} Equation~\eqref{eq:polynomialparameters} holds with constants depending only on $s,M$. Thus Proposition~\ref{prop:highpair} applies despite the polynomial growth of the rescaled mass, and Proposition~\ref{prop:energy} gives~\eqref{eq:closedenergy}. Choose one sufficiently large dyadic $N=N(s,M)$, valid for every datum in the ball and every interval on which the energy bound holds, such that
\[
 C_s\left(\frac{\ell_N^3}{N}+\frac{\ell_N}{N^2}\right)\le\frac{1}{4}.
\]
Apply~\eqref{eq:closedenergy} on $[0,T]$ before the first exit from $E(I_Nu_\lambda)<1$. The improved bound $E(I_Nu_\lambda(t))\le\frac{3}{4}$ excludes a finite exit time by continuity. Apply the same argument to $\overline{u_\lambda(-t)}$ to obtain
\[
 \sup_{t\in\R}E(I_Nu_\lambda(t))\le\frac{3}{4}.
\]
\emph{Step 3: Return to the original solution.} Mass conservation and the modified gradient bound give the low-frequency part of the $H^s$ estimate by interpolation. Above $N$, the definition of $I_N$ recovers the high-frequency part:
\[
 \|(u_\lambda)_{\le N}\|_{\dot H^s}
 \lesssim_s\mu_\lambda^{1-s},\qquad
 \|(u_\lambda)_{>N}\|_{\dot H^s}
 \lesssim_s N^{s-1}.
\]
Returning to the original variables and using mass conservation gives
\begin{equation}\label{eq:returnHs}
 \sup_{t\in\R}\|u(t)\|_{H^s}
 \lesssim_s M+\lambda^{\frac{1}{2}-s}
       (\mu_\lambda^{1-s}+N^{s-1})\le C_{s,M}.
 \qedhere
\end{equation}
\end{proof}

Local well-posedness and stability~\cite{Cazenave,CW,TaoBook} extend the estimate to arbitrary $H^s$ data. Approximate $u_0$ by Schwartz data in one fixed $H^s$ ball. Proposition~\ref{prop:globalHs} gives a common solution bound and hence a common local existence time. Iterating stability on compact intervals gives convergence in $C_tH^s_x\cap L^2_tW^{s,6}_x$. The limit is the unique global solution in $C(\R;H^s)\cap L^2_{\rm loc}(\R;W^{s,6})$. It satisfies~\eqref{eq:uniformball} and depends continuously on the datum in $C([-T,T];H^s)\cap L^2([-T,T];W^{s,6})$ for every $T>0$.

\subsection{Spacetime bounds and scattering}
The uniform $H^s$ bound and interaction Morawetz yield, by interpolation, a finite $L^p_tL^q_x$ norm with $2/p+3/q=1$. We partition the time axis into intervals on which this norm is small, absorb the cubic term in the Strichartz estimate, and control the Duhamel tails. For $\sigma\ge0$, set
\[
 S^\sigma(J')=L^\infty(J';H^\sigma)\cap L^2(J';W^{\sigma,6}),
 \qquad \|u\|_{S^\sigma(J')}=\|\la\nabla\ra^\sigma u\|_{S^0(J')}.
\]
For $\frac{1}{2}<s<1$, define
\begin{equation}\label{eq:criticalmixed}
 q_s=\frac6{3-2s},\qquad
 \theta=\frac{s-\frac{1}{2}}{s-\frac{1}{4}},\qquad
 \frac1p=\frac\theta4,\qquad
 \frac1q=\frac\theta4+\frac{1-\theta}{q_s}.
\end{equation}
Thus $p>6$ and $2/p+3/q=1$. To estimate the cubic term with
two factors in $L^p_tL^q_x$, we use the admissible pair
\begin{equation}\label{eq:derivativepair}
 a_* = \frac{2p}{p-2},\qquad b_* = \frac{6p}{p+4},\qquad
 \frac2{a_*}+\frac3{b_*}=\frac{3}{2},\qquad 2<a_*<3.
\end{equation}
We write $Z(J';u)=\|u\|_{L^p_tL^q_x(J')}$.

\begin{lemma}\label{lem:uniformS}
Let $\frac{1}{2}<s<1$ and $M\ge2$. The global solution with
$\|u_0\|_{H^s}\le M$ satisfies
\begin{equation}\label{eq:uniformS}
 \|u\|_{S^s(\R)}+Z(\R;u)
 +\|\la\nabla\ra^sF(u)\|_{L^{a_*'}_tL^{b_*'}_x(\R)}
 \le C_{s,M}.
\end{equation}
\end{lemma}
\begin{proof}
Pass the interaction Morawetz estimate~\eqref{eq:introInteractionMorawetz} from smooth approximants to $u$ by Fatou's lemma on compact intervals, then exhaust the time axis. Together with Proposition~\ref{prop:globalHs}, this gives
\[
 \|u\|_{L^4_{t,x}(\R^{1+3})}^4
 \lesssim\|u\|_{L^\infty_tL^2_x}^2
          \|u\|_{L^\infty_t\dot H^{\frac{1}{2}}_x}^2\le C_{s,M}.
\]
Interpolation with the Sobolev bound in $L^\infty_tL^{q_s}_x$ yields
$Z(\R;u)\le C_{s,M}$. The fractional product rule
\cite[Theorem~1]{GO} gives
\begin{equation}\label{eq:scatteringforcing}
 \|\la\nabla\ra^sF(u)\|_{L^{a_*'}_tL^{b_*'}_x(J')}
 \lesssim_s Z(J';u)^2
      \|\la\nabla\ra^su\|_{L^{a_*}_tL^{b_*}_x(J')},
\end{equation}
since $1/a_*'=2/p+1/a_*$ and $1/b_*'=2/q+1/b_*$.
Partition both half-lines, starting at zero, into at most
$C_{s,M}$ intervals $J_j$ on which $Z(J_j;u)\le\epsilon_s$.
Estimates \eqref{eq:homogeneousStrichartz}--\eqref{eq:linearStrichartz}
and \eqref{eq:scatteringforcing} imply
\[
 \|u\|_{S^s(J_j)}
 \le C_s\|u(t_j)\|_{H^s}
       +C_s\epsilon_s^2\|u\|_{S^s(J_j)},
\]
where $t_j$ is the endpoint nearer zero. Choose $\epsilon_s$ so that the final term is absorbed. On an unbounded terminal interval, first apply the estimate on finite restrictions, then pass to the limit. Summing the squared $L^2_tW^{s,6}_x$ norms gives global $S^s$ control, and~\eqref{eq:scatteringforcing} yields the nonlinear estimate in~\eqref{eq:uniformS}.
\end{proof}

\begin{proof}[Completion of the proof of Theorem~\ref{thm:main}]
\emph{Step 1: Construct the scattering states.} Since $a_*'<\infty$, the forcing bound~\eqref{eq:uniformS} and dual Strichartz~\eqref{eq:dualStrichartz} applied to $\langle\nabla\rangle^sF(u)$ show that the following integrals converge in $H^s$. Define
\[
 u_+=u_0-\ii\int_0^\infty S(-t)F(u(t))\dd t,\qquad
 u_-=u_0+\ii\int_{-\infty}^0 S(-t)F(u(t))\dd t
\]
and apply the same estimate on the time tails to obtain~\eqref{eq:scattering}.

\emph{Step 2: Control differences of solutions.} Let $u,v$ satisfy $\|u(0)\|_{H^s},\|v(0)\|_{H^s}\le M$, and set $w=u-v$. With $r_*=6p/(3p-4)$, the pair $(p,r_*)$ is admissible and $1/q=1/r_*-\frac{1}{6}$. Sobolev embedding gives
\begin{equation}\label{eq:Zdifference}
 Z(J';w)\lesssim
 \||\nabla|^{\frac{1}{2}}w\|_{L^p_tL^{r_*}_x(J')}
 \lesssim_s\|w\|_{S^s(J')}.
\end{equation}
Since $F(u)-F(v)=|u|^2w+v\bar u w+v^2\bar w$, the fractional
product rule yields
\begin{equation}\label{eq:nonlineardifference}
 \begin{aligned}
 &\|\langle\nabla\rangle^s(F(u)-F(v))\|_{L^{a_*'}_tL^{b_*'}_x(J')}\\
 &\quad\lesssim_s
 \bigl(Z(J';u)^2+Z(J';v)^2\bigr)\|w\|_{S^s(J')}\\
 &\qquad+
 \bigl(Z(J';u)+Z(J';v)\bigr)
 \bigl(\|u\|_{S^s(J')}+\|v\|_{S^s(J')}\bigr)Z(J';w).
 \end{aligned}
\end{equation}
Fix the bounds in~\eqref{eq:uniformS} before choosing $\epsilon$. Partition each half-line, starting at zero, into finitely many intervals $J_j$ where $Z(J_j;u)+Z(J_j;v)\le\epsilon$. Their number depends on $s,M,\epsilon$, while the $S^s$ norms are bounded independently of $\epsilon$. Apply~\eqref{eq:homogeneousStrichartz}--\eqref{eq:linearStrichartz} to $w$ and use~\eqref{eq:Zdifference} and~\eqref{eq:nonlineardifference}:
\[
 \|w\|_{S^s(J_j)}\le C_s\|w(t_j)\|_{H^s}
       +C_s(\epsilon^2+\epsilon C_{s,M})\|w\|_{S^s(J_j)},
\]
where $t_j$ is the endpoint nearer zero. Choose $\epsilon=\epsilon(s,M)$
so that the last term can be absorbed. Iteration over the finite
partition and summation of \eqref{eq:nonlineardifference} give
\[
 \|w\|_{S^s(\R)}+
 \|\langle\nabla\rangle^s(F(u)-F(v))\|_{L^{a_*'}_tL^{b_*'}_x(\R)}
 \lesssim_{s,M}\|u(0)-v(0)\|_{H^s}.
\]
Use finite restrictions on unbounded terminal intervals before passing to the limit. Applying dual Strichartz~\eqref{eq:dualStrichartz} to $\langle\nabla\rangle^s(F(u)-F(v))$ in the formulas for $u_\pm-v_\pm$ proves continuity of both scattering-state maps.

\emph{Step 3: Propagate higher regularity.} For $s\ge1$, take $s_0=\frac{3}{4}$ in~\eqref{eq:criticalmixed}. This gives the global $L^8_tL^4_x$ bound, uniform on bounded $H^s$ balls. The fractional product rule gives
\[
 \|\langle\nabla\rangle^sF(u)\|_{L^{\frac{8}{5}}_tL^{\frac{4}{3}}_x(J')}
 \lesssim_s \|u\|_{L^8_tL^4_x(J')}^2
 \|\langle\nabla\rangle^su\|_{L^{\frac{8}{3}}_tL^4_x(J')}.
\]
Since $(\frac{8}{3},4)$ is admissible, Strichartz estimates on a finite partition into intervals with small $L^8_tL^4_x$ norm, followed by iteration from the initial datum, give a uniform $S^s(\R)$ bound. The preceding tail and difference arguments then give $H^s$ scattering and continuity of the flow and scattering states~\cite{Cazenave,TaoBook}. This completes the proof of Theorem~\ref{thm:main}.
\end{proof}

\section*{Acknowledgments}
The second author is deeply grateful to his Ph.D. advisor, Prof. Benjamin Dodson, for his generous mentorship and guidance over many years. Their discussions since the author's doctoral studies have helped shape his understanding of and perspective on this line of research. The second author also thanks Prof. Nikolay Tzvetkov for helpful suggestions.

Q. Su was supported by the State Key Laboratory of Mathematical Sciences,
Academy of Mathematics and Systems Science, Chinese Academy of Sciences,
and the National Natural Science Foundation of China
(grants 12494541 and 12421001). Z. Zhao was supported by the National Key R\&D Program of China
(grant 2025YFA1018500), the National Natural Science Foundation of China
(grants 12426205 and 12271032), the Beijing Natural Science Foundation
(grant 1262019) and Beijing Institute of
Technology Research Fund Program for Young Scholars.\vspace{1mm}

\paragraph{Data availability.}
No datasets were generated or analyzed for this article.\vspace{1mm}

\paragraph{Competing interests.}
The authors declare no competing interests.\vspace{1mm}

\paragraph{Use of AI}
The main ideas of this work, including the wave-packet viewpoint
and the bilinear control of frequency-localized components, were
developed by the authors over an extended period. Discussions with
ChatGPT were helpful in the course of this work, and ChatGPT was
also used for language editing, LaTeX formatting, bibliographic
checks, and consistency checks during the preparation of the
manuscript. The authors critically reviewed and revised all
content and take full responsibility for the manuscript.

\end{document}